\documentclass{article}

\pdfoutput=1  

\usepackage{geometry} 

\usepackage[hyphens]{url}

\usepackage{hyperref}

\usepackage{amsmath}
\usepackage{amsthm,amsfonts,amssymb}  

\usepackage{cleveref,bm,bbm}

\usepackage{graphicx,wrapfig}
\usepackage{algorithmic}

\usepackage{textgreek,upgreek,bm}

\def\email#1{\href{mailto:#1}{#1}}

\graphicspath{{figures/}}

\newtheorem{theorem}{Theorem}[section]
\newtheorem{corollary}[theorem]{Corollary}
\newtheorem{proposition}[theorem]{Proposition}
\newtheorem{lemma}[theorem]{Lemma}

\Crefname{corollary}{Corollary}{Corollaries}
\Crefname{eqnarray}{eq.}{eqs.}
\Crefname{equation}{eq.}{eqs.}

\Crefname{figure}{Figure}{Figures}
\Crefname{tabular}{Tab.}{Tabs.}
\Crefname{table}{Tab.}{Tabs.}
\Crefname{theorem}{Theorem}{Theorems}
\Crefname{definition}{Definition}{Definitions}
\Crefname{section}{Section}{Sections}
\Crefname{subsection}{Section}{Sections}
\Crefname{proposition}{Proposition}{Propositions}
\Crefname{lemma}{Lemma}{Lemmas}
\Crefname{assumption}{Assumption}{Assumptions}
\Crefname{example}{Example}{Examples}

\usepackage{marginnote}

\reversemarginpar

\makeatletter
\newcommand\gobblepars{%
    \@ifnextchar\par%
 {\expandafter\gobblepars\@gobble}%
{}}
\makeatother

\def\wham#1{\smallbreak\pagebreak[3]%
\noindent\textbf{#1}\ \ \gobblepars}

\def\UplambdaN{\Uplambda_{\text{\scriptsize\sf n}}}

\def\fee{\upphi}
\def\tilfee{\widetilde{\upphi}}

 \def\tilU{\widetilde{U}}
 \def\tilX{\widetilde{X}}

 \def\tilclU{\widetilde{\mathcal{U}}}

\def\tilclX{\widetilde{\mathcal{X}}}
\def\tilclY{\widetilde{\mathcal{Y}}}

\def\sq{\hbox{\rlap{$\sqcap$}$\sqcup$}}
\def\qed{\ifmmode\sq\else{\unskip\nobreak\hfil
\penalty50\hskip1em\null\nobreak\hfil\sq
\parfillskip=0pt\finalhyphendemerits=0\endgraf}\fi\medskip}

\newcounter{rmnum}
\newenvironment{romannum}{\begin{list}{{\upshape (\roman{rmnum})}}{\usecounter{rmnum}
\setlength{\leftmargin}{12pt}
\setlength{\itemindent}{12pt}
\setlength{\rightmargin}{5pt}
\setlength{\itemsep}{2pt}
}}{\end{list}}

\def\bfmath#1{{\mathchoice{\mbox{\boldmath$#1$}}%
		{\mbox{\boldmath$#1$}}%
		{\mbox{\boldmath$\scriptstyle#1$}}%
		{\mbox{\boldmath$\scriptscriptstyle#1$}}}}

\def\bfmp{p}

\def\bfmy{\bfmath{y}}

\def\bfmD{\bfmath{D}}

\def\bfmS{\bfmath{S}}  
  
\def\bfmU{\bfmath{U}}

\def\bfmX{\bfmath{X}}
\def\bfmY{\bfmath{Y}}

\DeclareBoldMathCommand\bfmH{{H}}
\DeclareBoldMathCommand\bfmM{{M}}
\DeclareBoldMathCommand\bfmZ{{Z}}

\DeclareBoldMathCommand\bfmclM{{\cal M}}
\DeclareBoldMathCommand\bfmell{\ell}

\DeclareBoldMathCommand\bflambda{\lambda}
\DeclareBoldMathCommand\bfnu{\nu}

\DeclareBoldMathCommand\bfbeta{\beta}
\DeclareBoldMathCommand\bfkappa{\kappa}
\DeclareBoldMathCommand\bfzeta{\zeta}

\DeclareBoldMathCommand\bfPhi{\Phi}
\DeclareBoldMathCommand\bfPsi{\Psi}

\def\bfmp{p}

\DeclareBoldMathCommand\bfmr{r}
\DeclareBoldMathCommand\bfmy{y}

\def\U{{\sf U}}
\def\S{{\sf S}}

\def\state{{\sf X}}

\def\st{s}

\def\Prob{{\sf P}}

\def\Expect{{\sf E}}

\def\transpose{{\hbox{\it\tiny T}}}

\def\clG{\mathcal{G}}

\def\clN{\mathcal{N}}
\def\clS{\mathcal{S}}
\def\clT{\mathcal{T}}
\def\clU{\mathcal{U}}

\def\outpt{\mathcal{Y}}
\def\haoutpt{\widehat{\mathcal{Y}}}
\def\baroutpt{\overline{\mathcal{Y}}}

\def\barh{\bar{h}}
\def\barH{\bar{H}}

\def\Tg{G}

\def\outpt{\text{\footnotesize$\mathcal{Y}$}}
\def\haoutpt{\widehat{\text{\footnotesize$\mathcal{Y}$}}} 
\def\baroutpt{\overline{\text{\footnotesize$\mathcal{Y}$}}} 

\def\epsy{\varepsilon}

\def\varble{\,\cdot\,}
\def\eqdef{\mathbin{:=}}
\newcommand{\field}[1]{\mathbb{#1}}

\def\Re{\field{R}}

\def\ind{\field{I}}

\newcommand{\lgl}{\langle\,}
\newcommand{\rgl}{\,\rangle}

\newcommand{\clC}{\mathcal{C}} 
\newcommand{\clD}{\mathcal{D}} 
\newcommand{\clF}{\mathcal{F}} 
\newcommand{\clH}{\mathcal{H}} 
\newcommand{\clK}{\mathcal{K}} 
\newcommand{\clL}{\mathcal{L}}

\newcommand{\clR}{\mathcal{R}}

\newcommand{\clDV}{D_\infty}

\def\DVrate{\clR}

\renewcommand{\L}{\Lambda}

\def\limsup{\mathop{\rm lim{\,}sup}}

\def\argmin{\mathop{\rm arg{\,}min}}
\def\argmax{\mathop{\rm arg{\,}max}}

\def\vx{\vec{x}}
\def\vX{\vec{X}}

\def\har{\hat{r}}

\def\tilnu{\widetilde{\nu}}

\def\chlambda{\check{\lambda}}
\def\spmf{\hat{\nu}}
\def\pwr{\varrho}

\usepackage{enumitem}
\setlist[enumerate]{leftmargin=.5in}
\setlist[itemize]{leftmargin=.5in}

\title{Kullback-Leibler-Quadratic Optimal Control\thanks{The authors acknowledge support  from the National Science Foundation grant EPCN 1935389, and French National Research Agency grant ANR-22-PETA-0044.}}

\author{Neil Cammardella\thanks{Department of Electrical and Computer Engineering, University of Florida, USA (\email{ncammardella@ufl.edu})}
	\and 
	Ana Bu\v si\'c\thanks{Inria and DI ENS, \'Ecole Normale	Sup\'erieure, CNRS, PSL Research University, Paris, France (\email{ana.busic@inria.fr}, \url{https://www.di.ens.fr/\textasciitilde busic/})}	
	\and 
	Sean Meyn\thanks{Department of Electrical and Computer Engineering, University of Florida, and Inria International Chair, Paris (\email{meyn@ece.ufl.edu}, \url{http://www.meyn.ece.ufl.edu/})}
}

\usepackage{amsopn}

\begin{document}

\maketitle

\begin{abstract}
This paper presents approaches to mean-field control, motivated by distributed control of multi-agent systems.      Control solutions are based on a convex optimization problem, whose domain is a convex set of  probability mass functions (pmfs).   The main contributions follow:

\wham{1.}   
 Kullback-Leibler-Quadratic (KLQ) optimal control  is a special case,
  in which the objective function is composed of a control cost in the form of Kullback-Leibler divergence between a candidate pmf and the nominal, 
 plus a quadratic cost on the sequence of marginals.  Theory in this paper extends prior work on deterministic control systems, establishing that the optimal solution is  an exponential tilting of the nominal pmf.    Transform techniques are introduced to reduce complexity of the KLQ solution,   motivated by the need to consider time horizons that are much longer than the inter-sampling times required for reliable control. 
 
\wham{2.}  
Infinite-horizon KLQ  leads to a state feedback control solution with attractive properties. 
It can be expressed  as
	either
 state feedback, in which the state is the sequence of marginal pmfs, 
 	or an open loop solution is obtained that is more easily computed.

\wham{3.}    Numerical experiments are surveyed in an application of distributed control of residential loads to provide grid services, similar to utility-scale battery storage.   
The results show that KLQ optimal control enables the aggregate power consumption of a collection of flexible loads to track a time-varying reference signal, while simultaneously ensuring each individual load satisfies its own quality of service constraints.	

\wham{Keywords:} Mean field games,
distributed control, Markov decision processes, Demand Dispatch.
AMS:  90C40,  
	93E20,  	
	90C46	
   93E35,  	
   60J20,  	

\end{abstract}

\section{Introduction} \label{s:intro}

The goal of this paper is to obtain control solutions for mean-field models.  The optimization problems considered  are generalizations of standard  Markov Decision Process (MDP) objectives, in  both finite-horizon and average-cost settings.

\subsection{Mean field control} 
\label{s:intro_dc}

The mean-field control problem is an approach to distributed control of a collection of $\clN$   homogeneous ``agents'',
with $\clN\gg 1$,  modeled as discrete-time stochastic systems, with state processes at time $k$ denoted  $\{ X^i_k : 1\le i\le \clN \}$.  To avoid a long detour on notation it is assumed that the common state space $\state$ is finite.  
\begin{subequations}
For a single value $k$ and time horizon $K\ge 1 $, the  empirical distributions are denoted
\begin{align}
p^{\, \clN} (\vx)  &= \frac{1}{\clN} \sum_{i=1}^\clN \ind\{ (X_0^i ,\dots, X_K^i) = \vx\}    &&  \vx\in\state^{K+1}
\\
\nu^{\, \clN}_k(x)  &= \frac{1}{\clN} \sum_{i=1}^\clN \ind\{ X_k^i = x\}      \,,  &&    x\in\state \, ,
\label{e:empDist}
\end{align} 
where $\vx = (x_0,\dots,x_K)$ denotes an arbitrary element of $\state^{K+1}$.
 The set of pmfs on $\state^{K+1}$ is denoted by $\clS(\state^{K+1})$ for $K\ge 1$,    and $\clS(\state)$ for $K=0$.  
 
 \end{subequations}

The integer $\clN$ is regarded as a parameter in mean-field theory,  and assumptions imply that there is convergence as $\clN\to\infty$,
\[
\lim_{\clN\to\infty} p^{\, \clN}(\vx)   = p(\vx)\,,  \qquad  
\lim_{\clN\to\infty} \nu^{\, \clN}_k(x_k)   = \nu_k(x)   \,,     
\]
where $\nu_k\in \clS(\state)$ is the $k$th marginal of $p\in \clS(\state^{K+1})$ for $0\le k\le K$.    

In this paper this limit is achieved by assuming homogeneity of the statistics of each agent:
for  each $i$ the state evolution is consistent with $p$:
\begin{equation}
\Prob\{ X_{k+1}^i = x_{k+1} \mid (X_0^i ,\dots, X_k^i) =  \vx_0^{\, k}  \}    =   p( x_{k+1} \mid  \vx_0^{\, k} )
\label{e:GenMFMevolution}
\end{equation}
where the conditional pmfs are obtained from Bayes rule.

The paper concerns design of $p$ to  balance two objectives,  based on    a reference signal  $\{r_k\}$, and function $\outpt\colon\state\to\Re$:
\begin{romannum}
\item    $\nu_k\sim \nu_k^0$,  where $\{\nu_k^0\}$ models \textit{nominal behavior}.
\item  $\displaystyle \langle \nu_k,\outpt\rangle  \eqdef  \sum_{x\in\state} \nu_k(x)  \outpt(x)   \approx r_k$.
\end{romannum}

The agents considered in \Cref{s:num} represent a population of residential water heaters, and  $\outpt\colon \state \to \Re_+$ is chosen so that  $\langle \nu^{\, \clN}_k ,\outpt\rangle $ is the average power consumption over the population of loads.    

Two approaches to design are developed in this paper.   

\wham{Feedforward control:}  
A sequence $\{\clC_k : 1\le k\le K \}$ of real-valued cost functions  on the marginals is specified, and $p^*$ is obtained as the solution to     
\begin{equation}
J^\star(\nu_0^0) = \min_p \sum_{k=1}^K \clC_k(\nu_k)   
\label{e:GenMFC}
\end{equation}
where the minimum is over all pmfs with first marginal $\nu_0^0$. 
The two goals motivate the following objective function,  
\begin{equation}
\clC_k(\nu)  =  
\clD(\nu, \nu_k^0) + \frac{\kappa}{2} \big[\lgl \nu, \outpt \rgl - r_k\big]^2  \,, \qquad\nu\in \clS(\state) \,,
\label{e:Cost2}
\end{equation}
in which  $\kappa>0$ is a penalty parameter, and $\clD$ penalizes deviation from nominal behavior. The finite-horizon optimal control problem is thus
\begin{equation}
J^\star(\nu_0^0) = \min_p \sum_{k=1}^K \biggl[ \clD(\nu_k,\nu_k^0) + \frac{\kappa}{2} \big[\lgl \nu_k, \outpt  \rgl - r_k\big]^2  \biggr]
\label{e:DCMFG}
\end{equation}

It is envision that this finite horizon optimal control problem will be a component of a  model predictive control (MPC) strategy,  with time horizons for computation updates dictated by performance requirements and model accuracy.       
 
\wham{Feedback control:}    If the nominal model is Markovian, then the evolution of the marginals follow the dynamics of a controlled nonlinear state space model,
\begin{equation}
 \nu_{k+1} = f_k(\nu_k,\fee_k) \,,  \qquad k\ge 0\,,  \ \ \nu_0^0 \ \textit{given}
\label{e:GlobalPerspective}
\end{equation}
where  $\{ \fee_k \}$ is the input sequence, evolving on an abstract set  $\Upphi$.      A feedback policy takes the form $\fee_k = \clK_k(\nu_k)$.

Design choices for $\clK_k$ are proposed based on an infinite-horizon  solution of \eqref{e:DCMFG}.     Justification requires further assumptions, including time-homogenous dynamics for \eqref{e:GlobalPerspective},   which holds if the nominal model is a time-homogeneous Markov chain.

\subsection{MDPs and mean-field control}

The Markovian assumption for the nominal model is based on the standard controlled Markov chain model used in MDPs.

The model considered here is specified by a state space denoted $\S$, input space $\U$, and we denote $\state\eqdef \S\times\U$ (assumed finite).    
The  joint state-input process is denoted $\bfmX =\{X_k = (S_k, U_k) :  k\ge 0\}$.     
In finite-horizon optimal control the model includes a sequence of controlled transition matrices $\{ T_k  : k\ge 0 \}  $  and cost functions  $\{ c_k  : k\ge 0 \}  $,  with   $c_k\colon\state\to\Re$ for each $k$.

The dynamics of $\bfmX = (\bfmS,\bfmU) = \{S_k, U_k : k\ge 0\}$  are determined by the transition matrices as follows.
It is assumed that $\bfmX$ is adapted to a filtration $\{\clF_k : k\ge 0\}$  (so that $X_k$ is $\clF_k$-measurable for each $k$), and 
\begin{equation}
\Prob\{S_{k+1} = \st' \mid  \clF_k ; \ S_k = \st \,,  \  U_k = u \}  = T_k(x, \st')  \,,\qquad  x=(\st,u) \in\state\,, \st'\in\S  
\label{e:Pu}
\end{equation}

The set of   functions from $\S$ to the simplex  $\clS(\U)$ is denoted $\Upphi$,  and we let $\fee$ denote a generic element of $\Upphi$.
The decision rule defining the input sequence is  assumed to be Markovian:      
\begin{equation}
	\Prob\{U_k = u \mid \clF_{k-1} ;  \ S_k= \st \} = \fee_k(u\mid \st) \,,\qquad x=(\st,u) \in\state
	\label{e:U_MFC}
\end{equation}
with   $\fee_k\in\Upphi$ for each $k$.

The finite-horizon optimal control problem of MDP theory is  a special case of \eqref{e:GenMFC}, in which  $\clC_k$ linear for each $k$;
in this case $\clC_k(\nu_k) = \langle \nu_k, c_k \rangle = \sum_{x\in\state} \nu_k(x) c_k(x)$  for each $k$, 
 and the sum on the right hand side of \eqref{e:GenMFC} may be expressed
\[
 \sum_{k=1}^K 
 \langle \nu_k, c_k \rangle = 
 \sum_{k=1}^K   \Expect[ c_k(X_k)]  \,,  \qquad     X_k  \sim \nu_k   \,,
\] 
where   $\bfmX$ evolves according to the controlled Markovian dynamics.
This interpretation is the first step in the linear programming (LP) approach to MDPs introduced by Manne \cite{bor02a,man60a}.
The second step is to recognize that the dynamics can be expressed as a sequence of linear constraints on the marginals,
\begin{equation}
\sum_{u'} \nu_k(\st',u') = \sum_{\st,u} \nu_{k-1}(\st, u) T_{k-1}(x, \st') \,, \quad \st'\in\S \, , \ 1 \le k \le K\,, \quad \textit{   $\nu_0^0$  given.
}
\label{e:Manne}
\end{equation}

Another special case of \eqref{e:GenMFC} is variance-penalized optimal control, for which  
$
 \clC_k(\nu_k) =  \langle \nu_k,  c \rangle + \kappa \bigl[ \langle \nu_k,  c^2 \rangle - \langle \nu_k,  c \rangle ^2  \bigr]
$, 
with $\kappa>0$ a penalty parameter.  The solution to the optimization problem \eqref{e:GenMFC}   can be expressed using a randomized state feedback policy of the form \eqref{e:U_MFC} 
   \cite{alt99,put14,CSRL}.

\subsection{Kullback-Leibler-Quadratic control}
\label{s:KLQ}
 
In this approach to feedforward control we choose      a Markovian model of the form (\ref{e:Pu},\ref{e:U_MFC}) to define nominal behavior:  
\begin{subequations}
for a collection $\{\fee_k^0\}\subset\Upphi$,
\begin{align}
\bfmp^0(\vx)  &= \nu_0^0(x_0) P^0_0(x_0,x_1) P^0_1(x_1,x_2) \cdots P^0_{K-1}(x_{K-1},x_K)
\label{e:nom_mod}
\\
P^0_k(x,x')  &=  T_k(x, \st')\fee_{k+1}^0(u'\mid \st') \,,\qquad x ,x'\in\state
\label{e:nomP}
\end{align}
\end{subequations}

Any other  $\{\fee_k\}\subset\Upphi$ defines  a Markov chain $\bfmX$ with transition matrices,
\begin{equation}
	P_k(x,x') \eqdef \Prob\{X_{k+1} = x' \mid X_k=x\} = T_k(x, \st') \fee_{k+1}(u' \mid \st') \,.     
\label{e:Pk}
\end{equation}
The marginals evolve according to linear dynamics, similar to \eqref{e:Manne}:
\begin{equation}
\nu_k = \nu_{k-1} P_{k-1} \,,\qquad 1\le k \le K   
\label{e:mfm}
\end{equation} 
 in which   $\nu_k$ is interpreted as an $n$-dimensional row vector, with $n = | \state |$.

We obtain a convex program by optimizing over $\{\nu_k\}$,  similar to    the LP approach of  \cite{man60a}.  
Scalar variables $\{\gamma_k\}$ are introduced to simplify the objective, in anticipation of a Lagrangian decomposition:
\begin{subequations} 
\label{e:primal}
\begin{align}
J^\star(\nu_0^0) \eqdef & \min_{\nu,\gamma}  \Big[ \sum_{k=1}^K
\clD(\nu_k ,\nu_k^0) + \frac{\kappa}{2} \sum_{k=1}^K \gamma_k^2   \Bigr]
	&
\label{e:j_star}
\\
& \text{s.t.} \ \
\gamma_k = \lgl \nu_k, \outpt  \rgl - r_k \,, \quad   &
\label{e:gamma_constraint}
\\
&\quad \sum_{u'} \nu_k (\st', u')  =  \sum_{\st,u} \nu_{k-1}(\st, u) T_{k-1}(x, \st')  \,, \quad   &1\le k\le  K
\label{e:dynamics_constraint}%
\end{align}%
\label{e:DCMFGb}
\end{subequations}

The \textit{relative entropy rate} is adopted as the cost of deviation:
\begin{equation}
\clD(\nu_k ,\nu_k^0) \eqdef \sum_{\st,u}   \nu_k(\st,u) \log\Bigl(\frac{\fee_k(u\mid \st)}{\fee_k^0(u\mid \st)}\Bigr)
\label{e:KLrate}
\end{equation}
The terminology is justified through the following steps. First, we have seen that any randomized policy gives rise to a pmf $\bfmp \in \clS(\state^{K+1})$ that is Markovian:
$$
\bfmp(\vx) = \nu_0^0(x_0) P_0(x_0,x_1) P_1(x_1,x_2) \cdots P_{K-1}(x_{K-1},x_K)
$$  
The \textit{relative entropy}   (Kullback-Leibler divergence) is the mean log-likelihood:
\begin{equation}
D(\bfmp\| \bfmp^0) = \sum L(\vx)\, \bfmp(\vx) 
\label{e:KL}
\end{equation}
where $L = \log(\bfmp/\bfmp^0)$ is an extended-real-valued function on $\state^{K+1}$.   The expression for $P_k$ in \eqref{e:Pk} and the analogous formula for $P^0_k$ using $\fee_{k+1}^0$ gives
\begin{equation}
\begin{aligned}
L(\vx) = \log\Bigl(\frac{\bfmp(\vx)}{\bfmp^0(\vx)}\Bigr) 
& =
\sum_{k=0}^{K-1}  \log\Bigl(\frac{P_k(x_k,x_{k+1})}{P^0_k(x_k,x_{k+1})}\Bigr)
& = 
\sum_{k=1}^{K} \log\Bigl(\frac{\fee_{k}(u_k\mid \st_k)}{\fee^0_{k}(u_k\mid \st_k)}\Bigr)
\end{aligned} 
\label{e:LLR}
\end{equation}
Consequently, $D(\bfmp\| \bfmp^0) =     \sum_{k=1}^{K}  \clD(\nu_k ,\nu_k^0)$.


The proof of \Cref{p:convexity}  is contained in \Cref{app:cvx}.

\begin{proposition}
\label{p:convexity}
With $\clD$ chosen as the relative entropy rate \eqref{e:KLrate},
the optimization problem \eqref{e:primal} is convex in $\{\nu_k, \gamma_k: 1 \le k \le K\}$. Furthermore, the linear constraints in \eqref{e:dynamics_constraint}  are equivalent to \eqref{e:mfm}.
\qed
\end{proposition}

\subsection{Motivation from linear systems theory}
\label{s:LQRmot}

The approach to feedback control proposed in \Cref{s:IPD-Q} begins with consideration of   the infinite-horizon KLQ problem.   This is tractable only subject to additional assumptions.  

It is assumed that   the nominal model is a time-homogeneous Markov chain, and that the reference signal is \textit{constant},  $r_k\equiv r$,  $k\ge 0$.   On optimizing for each $r\in\Re$ we obtain a continuous family of optimizers,  $\{ \fee^\star_k(u\mid \st;r) :   (\st,u)\in\state\,  ,  k\ge 0 \,  ,  \ r\in\Re\}$.
A potentially useful policy for tracking is then,
\begin{equation}
\fee_k(u\mid \st) = \fee^\star_k(u\mid \st;r_k) \,,\qquad     (\st,u)\in\state\, \  k\ge 0
\label{e:IPD-Qsoln}
\end{equation}
Motivation for this approach may be found in the theory of optimal control for linear systems.   

 Consider the linear system with $n$-dimensional state $\bfmX$, $m$-dimensional input $\bfmU$,  and scalar output $\bfmY$, evolving as
\begin{equation}
X_{k+1} = A X_k  + B U_k   +   N_{k+1}\,,   \quad   Y_k = C^\transpose X_k  +  W_{k+1}
\label{e:LQR}
\end{equation}
in which $\{ N_{k+1}  , W_{k+1} \}$ are i.i.d.,  mutually independent,  with zero mean and finite covariances.  
The cost is quadratic, $c(x,u;r) = (y-r)^2  +  u^\transpose R u$ with $R>0$.

The goal is to solve the average cost optimal control problem.   The solution is obtained via state-augmentation:   define $X^r_k = [X_k;  r_k]$, where $r_{k+1} = r_k=r$ defines the dynamics.    The solution is linear state feedback,
\begin{equation}
U_k =  - K^\star X_k  +  G^\star  r\,,\qquad k\ge 0\, ,
\label{e:LQG+ref}
\end{equation} 
where $[K^\star;G^\star]$ is the optimal  gain.   
The optimal gain   does \textit{not} depend on $r$ or the  distribution  of   $N_k$ or $ W_k$.   

The special case in which the disturbances are \textit{zero} is most closely related to the nonlinear control problem considered in  \Cref{s:IPD-Q}.    Consider the finite horizon  objective
\[
J^\star_K(x)  = \min_{\bfmU} 
\sum_{n=0}^K  c(X_k,U_k)  
		= \min_{u}  \bigl\{  c(x,u;r)   +   J^\star_{K-1}( A x  + B u )  \}\,,\qquad X_0 =x\in\Re^n\,.
\]
It is not useful to let $K\to\infty$ without modification,  since the cost $c(x,u;r)  $ is never zero.   This is why the relative value functions $h_K(x) = J^\star_K(x) -J^\star_K(0) $  are introduced,  which solve the Bellman equation in modified form
\[
\eta_K +
h_K(x)   =   \min_{u}  \bigl\{  c(x,u;r)   +   h_{K-1}( A x  + B u )  \}\,,\qquad X_0 =x\in\Re^n\,,
\]
with $\eta_K = J^\star_K(0) -J^\star_{K-1}(0) $.   As $K\to\infty$,  the pair $(\eta_K ,
h_K)$ converge to a solution to   the average cost optimality equation (ACOE),
\[
\eta^\star +
h^\star(x)   =   \min_{u}  \bigl\{  c(x,u;r)   +   h^\star( A x  + B u )  \}\,,\qquad x\in\Re^n\, ,
\]
whose minimizer is precisely \eqref{e:LQG+ref}.     The proof is standard, though usually presented in the purely stochastic setting.  It is especially simple in this LQR setting since each of the functions $\{h_N\}$ are quadratic   \cite{put14,CSRL}.

When $\bfmr$ is time varying, it is standard practice to apply the ``hack'' 
\begin{equation}
U_k =  - K^\star X_k  +  G^\star  r_k\,,\qquad k\ge 0\,.
\label{e:LQG+ref2}
\end{equation}
The most compelling motivation is found in  the deterministic, continuous time setting: under mild conditions, the Return Difference Equation tells us that the closed loop dynamics from reference input to output are \textit{passive} \cite{andmoo90}.  Passivity is lost for discrete time models, but can be expected to hold approximately when the discrete time model is obtained from sampling a continuous time system.

\subsection{Main results}
\label{s:MainResults}

The contributions of this paper fall into three categories:

\begin{subequations}

\wham{1. Feedforward control}

Consideration of the dual of the convex optimization problem \eqref{e:primal} leads to many insights.  The main conclusions summarized here are a 
 special case of \Cref{t:dualFunctional}:
\begin{theorem}
\label{t:opt1}   \emph{[KLQ solution].}
Consider the convex program \eqref{e:primal}.   An optimizer $\{\fee^\star_k: 1 \le k \le K\}$ exists, is unique, and is of the form:
\begin{align}
\qquad
\fee_k^\star(u\mid \st) 
&=
\fee_k^0(u\mid \st) \exp\bigl(\sum_{\st'} T_k(x, \st')g^\star_{k+1}(\st') + \lambda^\star_k \outpt(\st,u) - g_k^\star(\st) \bigr)\,,
\label{e:phistar1}
\\
   \text{\it where} \ \ 
g_k^\star (\st) &=
\log\Bigl( \sum_u \fee_k^0(u \mid \st) \exp\bigl(\sum_{\st'} T_k(x, \st')g^\star_{k+1}(\st') + \lambda^\star_k \outpt(\st,u)   \bigr) \Bigr)
\label{e:lambda1}
\end{align}
and $\{\lambda^\star_k: 1 \le k \le K\}$, $\{g^\star_k(\st): 1 \le k \le K\}$ are the Lagrange multipliers for the constraints \eqref{e:gamma_constraint} and \eqref{e:dynamics_constraint}, respectively, and $g_{K+1} \equiv 0$.
\end{theorem}

\label{e:KLQsoln}
\end{subequations}

\Cref{p:algo} motivates a two-step approach in which $\lambda^\star$ is obtained as the solution to a convex program that maximizes the dual function $\varphi^\star$,   and then $g^\star$ are computed  through the nonlinear recursion \eqref{e:lambda1}.   Hence the larger computational challenge is computing  $\lambda^\star$.      Expressions for the derivatives of  $\varphi^\star$  involve means and variances of $\outpt(X_k)$, which  invites the application of Monte-Carlo techniques when the state space is large or even uncountable---see   \Cref{s:alg}.

\wham{2. Feedback.}

   \Cref{s:IPD-Q} concerns control design following steps analogous to the approach used  in linear systems theory to obtain the feedback control strategy \eqref{e:LQG+ref2}.  
Justification of the average-cost optimality equation (ACOE) requires that we turn to a time-homogeneous model,  meaning that $T_k =T$ and  $\fee_k^0= \fee^0 $, independent of $k$.     

Even with $r_k\equiv r$ fixed,  the solution to \eqref{e:Cost2}  is not time homogeneous, but  on letting $K\to\infty$ the policies converge to a solution of an ACOE.  This is equivalently expressed as  the solution to a deterministic optimal control problem:
 \begin{equation}
 \text{System:}
 \
 \nu_{k+1} = f(\nu_k,\fee_k) \,,  \quad 
 \text{Cost:} 
 \
  c(\nu,\fee;r)  =  \clDV(\spmf,\fee)     
				+ \frac{\kappa}{2} [  \lgl \nu, \outpt \rgl - r ]^2
\label{e:IPDQ-global}
\end{equation}
where  the marginals   $\{\nu_k\}$  are viewed as a state process, evolving on the simplex $\clS(\state) $, and  $\fee_k\in \Upphi$ is regarded as an input.   
The system equation is of the form \eqref{e:mfm}, but simplified because of the time-homogeneity assumptions imposed here, giving
\[
f(\nu,\fee)\big|_{x'=(\st',u')} =  \sum_{x\in\state} \nu(x) T(x,\st')  \fee(u'\mid \st')
\] 
 Hence $f$ is bilinear in the pair $(\nu,\fee)$.   
  Identification and justification of the term $\clDV$ in \eqref{e:IPDQ-global} requires further notation and analysis.

Consider the infinite horizon objective,
\begin{equation}
\eta^\star( r) = 
 \min  \limsup_{K\to\infty} \frac{1}{K} \sum_{k=1}^K \biggl[ \clD(\nu_k,\nu^0_k) + \frac{\kappa}{2} \big[\lgl \nu_k, \outpt  \rgl - r\big]^2  \biggr]
\label{e:KLQinfty}
\end{equation}
in which the minimum is over all  $\{ \fee_k\}  \subset \Upphi$.    
The  following notational conventions are required to describe the structure of its solution:
\wham{(i)}  Any  $\fee\in\Upphi$ defines a transition matrix $P_\fee$,  and any   pmf $\pi$ that is invariant for $P_\fee$  admits the decomposition
\begin{equation}
\pi(\st,u) = \fee(u\mid \st) \spmf(\st)   
\label{e:pimu}
\end{equation}
where $\spmf$ is the steady-state pmf for $\bfmS$ under this policy.

\wham{(ii)} 
With $\fee$ and $\spmf$ as above, the steady-state  relative entropy rate is denoted
\begin{equation}
\clDV(\spmf,\fee)  \eqdef \sum_{\st,u}    \fee(u\mid \st)  \spmf(\st)  \log\Bigl(\frac{\fee(u\mid \st)}{\fee^0(u\mid \st)}\Bigr)
\label{e:KLrateSS}
\end{equation}

\begin{theorem}
\label{t:opt2}   \emph{[Infinite-horizon KLQ solution].}
Suppose that the nominal transition matrix $P^0$ has unique invariant pmf $\pi^0$, and fix any  $\kappa>0$ and $r\in\Re$.
Then,  there is a solution to \eqref{e:KLQinfty} in which $\fee_k^\star = \fee^\star$ for each $k$, obtained from the optimization problem  
\begin{equation}
\argmin_{\pi,\fee}  \bigl\{  \clDV(\spmf,\fee)     
				+ \frac{\kappa}{2} [  \lgl \pi, \outpt \rgl - r ]^2     :   \pi P_\fee = \pi  \bigr\}
\label{e:IPDQ}
\end{equation}
This   optimization problem 
is convex,  with unique solution $\{ \pi^\star,\fee^\star \} $.
\end{theorem}

The convex program \eqref{e:IPDQ} reduces to the ``IPD'' convex program of \cite{meybarbusyueehr15,busmey16v,chehasmatbusmey18} 
as $\kappa\to\infty$  (see discussion surrounding \eqref{e:IPD} in the literature review).   The two convex programs  
are differentiated by the introduction of a quadratic cost on the marginals, so  the policy $\fee^\star$ obtained from \eqref{e:IPDQ}   is henceforth called the \textit{IPD-Q solution}.
Much of \Cref{s:IPD-Q}   is devoted to obtaining approximations of this solution, as well as computational methods to obtain the exact solution.

\wham{a)   HJB solution and  LQR approximation.}

Viewed as a  deterministic optimal control problem, with system and cost given in  \eqref{e:IPDQ-global},  
another solution to  \eqref{e:KLQinfty} is obtained  as state feedback $\fee_k^\star = \clK^\star (\nu_k^\star, r)$,  for some mapping $\clK^\star \colon \clS(\state)\times\Re \to\Upphi$.    
The IPD-Q solution is obtained via  $\fee^\star = \clK^\star (\nu^\star; r)$,  with $\nu^\star$ the steady-state marginal of $\bfmS$ under $P_{\fee^\star}$.

 Because computation of $\clK^\star$ is complex if $|\state|$ is large, and in anticipation of finer analysis of the performance of this policy,  much of \Cref{s:IPDQ_HJB} is devoted to   ``small signal'' approximations.

Let $\{x^i =(\st^i,u^i) :  1\le i\le n\} $ be an enumeration of the state space $\state$, with $n=|\state|$.  
As a corollary to \Cref{t:IPDQ-LQRa,t:IPDQ-LQR}, coefficients $\{K^\star_{i,j}  , G^\star_i \}$ are constructed for which 
$\fee_k^\star(u^i \mid \st^i,r )  = \fee_k(u^i \mid \st^i,r )  +O(r^2)$ for each $i$,  with
\begin{equation}
\fee_k(u^i \mid \st^i,r ) \eqdef  \fee^0(u^i\mid \st^i) \exp \Bigl(  \frac{1}{\fee^0(u^i\mid \st^i)}  \Bigl(   - \sum_j K^\star_{i,j} \tilnu_k(x^j)   +  G^\star_i  r  \Bigr)      -    \Gamma ( \st^i,r ) \Bigr)  
\label{e:SmallSignalPolicy}
\end{equation}
$\tilnu_k(x^i) \eqdef \nu_k(x^i)   -\pi^0(x^i) $ for each $i$, and
 $\Gamma$ is a normalizing constant, defined so that $\fee_k(\varble \mid \st^i,r )$ is a pmf on $\U$ for each $\st^i,r$.

\wham{b)   Lagrangian relaxation.}
A Lagrangian relaxation   leads to a characterization of the IPD-Q solution in terms of a standard ACOE.   Similar to \eqref{e:gamma_constraint},  we introduce the variable $\gamma =
\lgl \nu_k, \outpt  \rgl - r$,   and let $\lambda^\star\in\Re$ denote the Lagrange multiplier associated with this scalar constraint;   it  is identified in \cref{e:rCS}
 as 
$\lambda^\star = \kappa [r - \langle \pi^\star, \outpt \rangle ]$.

The relative value function $h^\star$ that solves the ACOE provides a representation of the IPD-Q solution in \eqref{e:phistarIPDQ}:
\[
\fee^\star(u\mid \st) 
=
\fee^0(u\mid \st) \exp\bigl( \barh^\star (\st,u)  + \lambda^\star \outpt(\st,u)- \Gamma^\star(\st) \bigr)\,,
\]
with $\displaystyle \barh^\star(x) = \sum_{u'} T(x , \st')  h^\star(\st')$ and $\Gamma^\star(\st) $ a normalizing factor.

\wham{c)   ODE solution and small signal approximation.}

Rather than compute $ \lambda^\star$ for each $r$, it is argued that it is simpler to let $\lambda$ be the independent variable.    The family of relative value functions $\{ h^\lambda : \lambda\in\Re \}$ solve an ordinary differential equation, whose vector field is identified in \eqref{e:ODEsoln}.    In addition to offering a tool for exact computation, 
this leads to approximation of the IPD-Q solution.

\smallskip

These conclusions lead to several approaches to feedback control for tracking a time varying reference signal.   Remember, in the following 3 options, the family $\{\fee_k : k\ge 0\}$ is proposed for local decision making in a mean-field control architecture.   

\wham{1.}  The feedback solution \eqref{e:IPD-Qsoln}  using the collection  $\{ \fee^\star_k(\cdot \mid \cdot\, ;r) :    k\ge 0 \,  ,  \ r\in\Re\}$.

\wham{2.}   The open-loop strategy $\fee_k(\cdot \mid \cdot) =  \fee^\star(\cdot \mid \cdot\, ;r_k)$,   with
$\{\fee^\star(\cdot \mid \cdot\, ; r) :  r\in\Re\} $ the IPD-Q solutions.

\wham{3.}   In option 2 above,  it is assumed that $r_k$ is made available to each agent, at each time $k$, as an external control signal.   
 A refinement is obtained by designing a control signal $\{\zeta_k: k\ge 0\}$ based on filtering measurements, such as error feedback,
 \begin{equation}
\zeta_k = \sum_{i=0}^k g_{k-i}  e_i\,,  \quad k\ge 0\,, \qquad e_i = r_i -\langle \nu_i, \outpt \rangle
\label{e:ErrorFeedback}
\end{equation}
The randomized decision rule for each agent is then  $ \fee_k(\cdot \mid \cdot)= \fee^\star(\cdot \mid \cdot\, ;\zeta_k)$,  with
$\{\fee^\star(\cdot \mid \cdot\, ; \zeta) :  \zeta\in\Re\} $ the IPD-Q solutions.

The linearized dynamics described in \Cref{t:IPDQ-LQRb}  can aid in the design of the filter in \eqref{e:ErrorFeedback}.

\wham{3. Application to Demand Dispatch.}

The original motivation for the research surveyed here is application to distributed control of power systems.  The term  \textit{Demand Dispatch} was introduced in the conceptual article \cite{brolureispiwei10} to describe the possibility of distributed intelligence in electric loads, designed so that the population would help provide supply-demand balance in the power grid.

 The numerical results surveyed in \Cref{s:num} illustrate the application of KLQ to control a large population of residential loads.   As expected, tracking error  can be made arbitrarily small with large $\kappa>0$, provided the reference signal is feasible.

 It is found in numerical experiments that the histograms defining the state of the mean-field model rapidly ``forget'' their initial conditions. For example, \Cref{f:hist-evo} shows the evolution of the histograms over time from six different degenerate initial conditions; within a few hours, they become nearly identical. If this phenomenon holds under general conditions, then it has important implications for control design. Further discussion is contained in \Cref{s:con_arc}.

\begin{figure}[h]

\centering

\includegraphics[width=.9\textwidth]{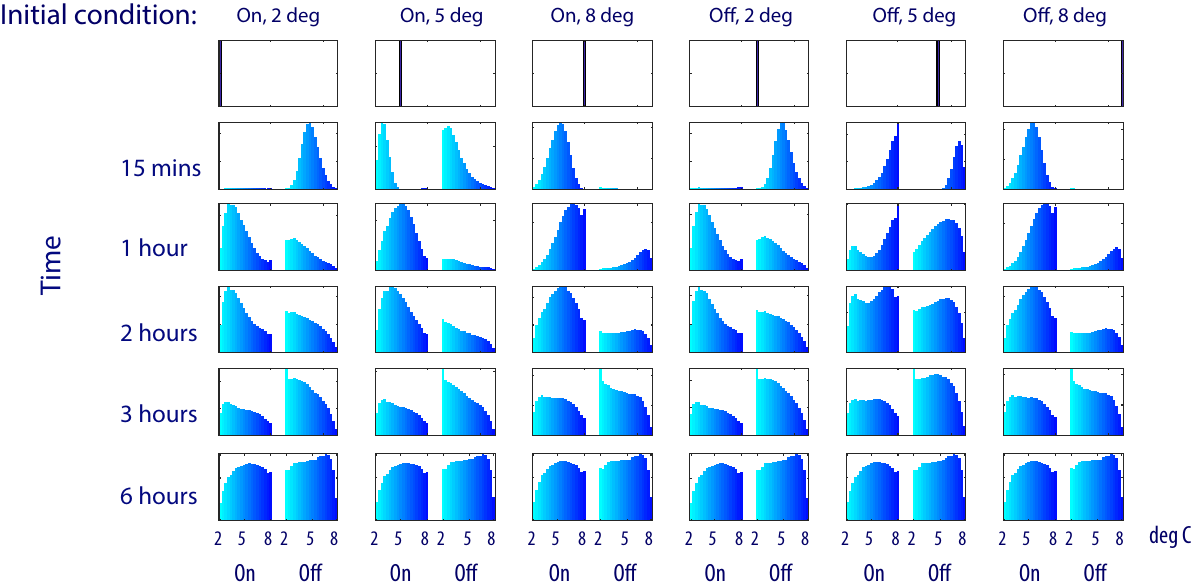}

\caption{\small
	Evolution of the marginals $\{ \nu_k^\star\}$ of individual agents with $\kappa=150$, from six different initial conditions.
	The histograms nearly coincide after about three hours.}
\label{f:hist-evo} 
\vspace{-.1cm}
\end{figure}

%
%

\subsection{Literature review}
\paragraph{Mean field control}

The optimization problem \eqref{e:GenMFC}
 is inspired by mean-field game theory \cite{lasry07mean,huacaimal07,huang06large,cai15,guelaslio11,yinmeymeysha12} (see \cite{cardel18a,cardel18b,cai21,tagmeh23} for recent surveys).

Mean-field control differs from mean-field game theory only because of greater control at the microscopic layer: we do not assume that an individual in the population is free to optimize based on its local objective function, so we avoid the fragility of Nash equilibria.   
This description is similar to \textit{ensemble control} in physics  (see \cite{li10} for history), and   many in the power systems area opt for this term rather than mean-field control (see \cite{cheche17b,chehasmatbusmey18} and their references).

\paragraph{Demand Dispatch}
  The goal of Demand Dispatch is to modify the behavior of loads so that their aggregate power consumption tracks a reference signal $\{r_k\}$  that is synthesized by a \textit{balancing authority} (BA).    Randomized control techniques have been proposed in \cite{matkoccal13, tintrostr15, meybarbusyueehr15,almfrohin17, cheche17b,bencolmal19} based on various control architectures.

The following control strategy is common to the approaches described in \cite{meybarbusyueehr15,chehasmatbusmey18}. It is assumed that a family of transition matrices $\{P_\zeta :\zeta \in \Re\}$ is available at each load. A sequence $\{\zeta_0, \zeta_1, \dots\}$ is broadcast from the BA, based on measurements of the grid, and at time $k$ the $i$th load transitions according to this law:
$$
\Prob\{X^i_{k+1} = x' \mid X^i_k =x,\ \zeta_k =\zeta \} = P_\zeta(x,x')
$$
The feedback solution \eqref{e:ErrorFeedback} was proposed in \cite{meybarbusyueehr15},  and tested in this and later research using  
$ e_i = r_i -\langle \nu^{\, \clN}_i, \outpt \rangle$ \cite{chehasmatbusmey18}.

\paragraph{IPD}
The paper \cite{meybarbusyueehr15} re-interprets the control solution of \cite{tod07} as a technique to create the family $\{P_\zeta\}$ through the solution to the nonlinear program:
\begin{equation}
\begin{aligned}
P_\zeta \eqdef 
\argmax  \  
& \bigl\{  \zeta  \lgl \pi,  \outpt \rgl  - \DVrate(P\| P^0)  \bigr\}  \,, \qquad\zeta\in\Re\,,
\end{aligned} 
\label{e:IPD}
\end{equation}
where $\DVrate$ denotes the rate function of Donsker and Varadhan  \cite{demzei98a,konmey05a},
\begin{equation}
\DVrate(P\| P^0) \eqdef \sum_{x,x'} \pi(x) P(x,x') \log\Bigl(\frac{P(x,x')}{P^0(x,x')} \Bigr)
\label{e:DVrate}
\end{equation}  
in which $\pi$ is the invariant pmf for $P$. The maximum in \eqref{e:IPD} is over all $(\pi, P)$ subject to the invariance constraint $\pi P = \pi$
\cite{meybarbusyueehr15, busmey16v}.   The convex program \eqref{e:IPD} is called the \textit{Individual Perspective Design} (IPD) in \cite{busmey16v}.

Hence IPD-Q may be interpreted as a new approach to designing $\{ P_\zeta \}$.

The finite-horizon version of \eqref{e:IPD} is also considered in \cite{meybarbusyueehr15,busmey16v}, similar to the KLQ formulation:
\begin{equation}
\begin{aligned}
\bfmp^\zeta & \eqdef \argmax_\bfmp \Bigl\{ \zeta \Expect_p\Bigl[ \sum_{k =1}^K  \outpt(x_k)  \Bigr] - D(\bfmp \| \bfmp^0) \Bigr\} \, .
\end{aligned}
\label{e:IPDfinite}
\end{equation}
Provided the entries of $T_k(x,\st)$ take on only binary values,
the finite-horizon IPD solution is obtained as a tilting of  the nominal model:
\begin{equation}
\bfmp^\zeta (\vx) = \bfmp^0(\vx) \exp\Bigl(\zeta \sum_{k=1}^K \outpt(x_k) - \Lambda(\zeta) \Bigr)\,,
\quad \textit{with $\Lambda(\zeta)$  a normalizing constant.   
}
\label{e:IPDsolnM}
\end{equation}

\paragraph{KLQ and optimal transport}  
Extensions of the KLQ objective will likely provide useful relaxations of the classical \textit{optimal transport} problem, in which the goal is to steer   $p^0$ to a given target pmf $p^\star$ 
\cite{vil08,peycut20,chegeopav20}.     
Rather than match the target pmf, we might match  $M$ generalized moments,  minimizing $D(\bfmp\| \bfmp^0)    $  subject to  $\langle \bfmp , \clG_i \rangle = 
\langle \bfmp^\star , \clG_i \rangle $ for each $i$,  with $\clG_i \colon\state^{K+1}\to\Re$.  

A special case is the tracking problem,
\begin{equation}
\begin{aligned}
&
\min_\bfmp  \bigl\{ 
D(\bfmp\| \bfmp^0)    
\quad \text{subject to} \ \Expect_p\bigl[\outpt(X_k) \bigr]    = r_k  \,, \ \  1\le k\le K \bigr\}
\end{aligned} 
\label{e:cc17}
\end{equation}
This optimization problem  is proposed in \cite[Section~5]{cheche17b},    along with the explicit solution  
\begin{equation}
\bfmp^\star (\vx) = \bfmp^0 (\vx) \exp\Bigl(\sum_{k =1}^K \beta_k \outpt(x_k) - \Lambda(\beta) \Bigr)
\label{e:cc17_soln}
\end{equation}
in which $\beta\in\Re^K$ are Lagrange multipliers corresponding to the $K$ constraints, and $\Lambda(\beta)$ a normalizing constant.

The convex program formulation \eqref{e:primal} has many advantages. First, \eqref{e:primal} is always feasible, while feasibility of \eqref{e:cc17} requires conditions on $\bfmp^0$ and $\{r_k\}$. \Cref{t:opt1} requires no assumptions on the model or reference signal. 
Flexibility in choice of $\kappa$ allows for \textit{learning} the characteristics of an ``expensive'' reference signal. It is anticipated that the penalty parameter $\kappa$ can be used to make tradeoffs between tracking performance and robustness to modeling error: robustness and sensitivity analysis will be a topic of future research.

 Finally,  as assumed to obtain the representation \eqref{e:IPDsolnM}, the formula 
 \eqref{e:cc17_soln} is meaningful only when  $T_k(x,\st)$ take on only binary values.   
 A goal of the research surveyed in this paper is to remove this restriction.   

\smallskip

The similarity between \eqref{e:IPDsolnM} and \eqref{e:cc17_soln} is not accidental, but follows from an  alternative interpretation of the  IPD design \eqref{e:IPDfinite}. For a scalar $r_0 \in \Re$, consider the constrained optimization problem
\begin{equation}
\begin{aligned}
&
\max_\bfmp  \bigl\{  - D(\bfmp \| \bfmp^0)        \bigr\} \,, \quad  \text{subject to} \   \Expect_p\Bigl[ \sum_{k =1}^K \outpt(x_k)  \Bigr]    = K r_0    \,, \quad 1\le k\le K
\end{aligned}
\label{e:IPDfinite-interpreted}
\end{equation}
The dual function $\varphi^\star\colon\Re\to\Re$ is defined by  
$$
\varphi^\star(\lambda)  = \max_\bfmp \Bigl\{ \lambda  \Expect_p\Bigl[ \sum_{k =1}^K \outpt(x_k) \Bigr] - D(\bfmp \| \bfmp^0) \Bigr\} - \lambda K r_0
$$
where $\lambda\in\Re$ is a Lagrange multiplier. It is evident that the optimizer $\bfmp^{*\lambda}$ is an IPD solution for each $\lambda$. Consequently, for each $\zeta$, the IPD solution \eqref{e:IPDfinite} also solves \eqref{e:IPDfinite-interpreted} for some scalar $r_0(\zeta)$.

\wham{Contributions} 

Most of the contributions were surveyed in \Cref{s:MainResults}.  
The main contribution of this paper is the discovery of  hidden convexity in the nonlinear program \eqref{e:DCMFGb},  
which leads to structure for the optimal solution in \Cref{t:opt1}.   
Properties of the dual surveyed in \Cref{t:dualFunctional} lead to computational techniques for this new class of optimal control formulations;  see  \Cref{p:algo}  and its corollary.     The application of these techniques to the  infinite-horizon setting in \Cref{s:IPD-Q} is novel, and the main results surveyed there are new.    

\smallskip
 Portions of the results reported here were summarized in the conference article \cite{cambusjimey19}.
 In this preliminary work,  the transition matrix $T_k$ was assumed \textit{deterministic}, so that all randomness arose from the randomized policy.  
 All of the results in this paper allow for general Markovian dynamics.
 
Extensions to resource allocation are summarized in \cite{cambusmey20}.   More on these topics may be found in the first author's PhD dissertation
\cite{NeilCammardellaThesis21}.

\wham{Organization} 
The remainder of this paper is organized as follows:  \Cref{s:klq} describes a relaxation technique motivated by the desire to reduce computational complexity, along with a full analysis of the convex program \eqref{e:primal} and its dual. 
\Cref{s:IPD-Q} contains extensions to the infinite-horizon setting.
Results from numerical experiments are collected together in \Cref{s:num}. 
Conclusions and directions for future research are contained in \Cref{s:conc}.

\section{Kullback-Leibler-Quadratic Optimal Control}
\label{s:klq}

\subsection{Subspace relaxation}
\label{s:sub}

A relaxation of the convex program \eqref{e:primal} is described here. Motivation is most clear from consideration of distributed control of a collection of residential water heaters. These loads are valuable as sources of virtual energy storage since they in fact are energy storage devices (in the form of heat rather than electricity), and are also highly flexible. Flexibility comes in part from their  extremely non-symmetric behavior: a typical unit may be on for just five minutes, and off continuously for more than six hours. The inter-sampling time at the load should be far less than five minutes to obtain a reliable model for control.

\begin{wrapfigure}[14]{R}[-2pt]{0.4\textwidth}
\vspace{-.4cm}

 		\centering
		\includegraphics[width=.95\hsize]{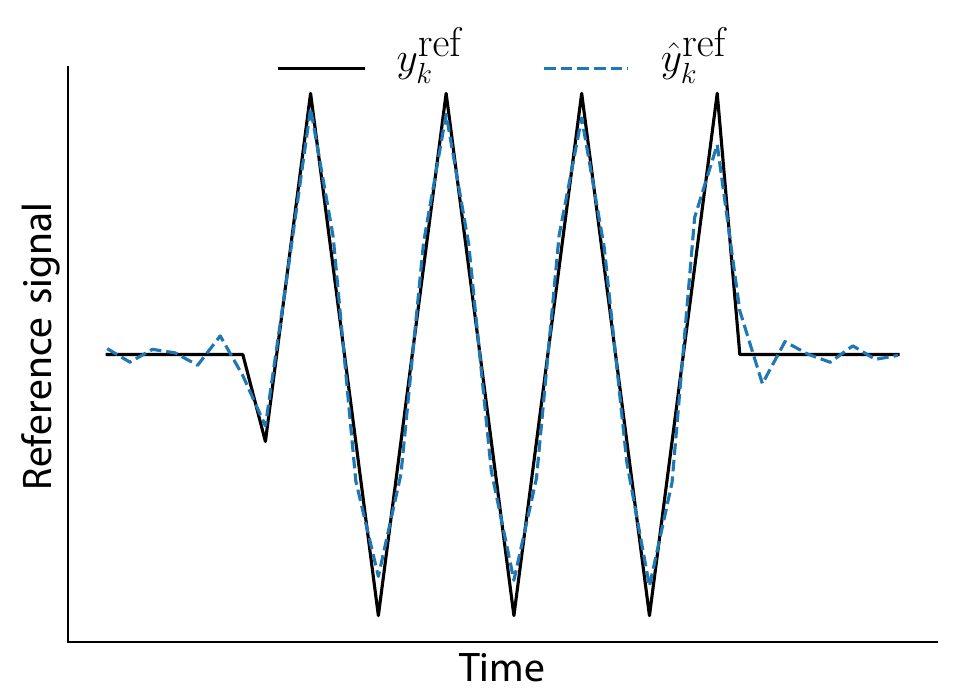}
				\vspace{-1em}
 		\caption{\small
	The scaled reference signal $ y^{\textrm{ref}}_k = r_k/\pwr $   and its transform.}
		\label{f:tri} 
\end{wrapfigure} 
	
On the other hand, it is valuable for the time horizon to be on the order of several hours. For example, peak-shaving is more effective when water heaters have advance warning to pre-heat the water tanks. To obtain a useful control solution will thus require a very large value of $K$ in \eqref{e:primal}. To reduce complexity, an approach is proposed here based on lossy compression of $\{r_k\}$ using transform techniques.



The transformations are based on a collection of functions $\{w_n:  1 \le n \le N \}$, with $w_n\colon \{0,1,\dots,K\} \to \Re$ for each $n$, and $N\ll K$. The transformed signal is the $N$-dimensional vector $\har$ with  $\har_n = \sum_k w_n(k) r_k$ for each $n$, and the transformed function on $\state^{K+1}$ is denoted 
\begin{equation*}
\haoutpt_n(\vx) = \sum_{k=1}^K  w_n(k) \outpt (x_k)  \,, \qquad   1\le n\le N  
\label{e:fn}
\end{equation*}

The goal is to achieve the approximation $\lgl \bfmp, \haoutpt_n \rgl \approx \har_n$ for each $n$, while maintaining $\bfmp \approx \bfmp^0$. For example, a Fourier series can be used,  with  frequency $\omega>0$, and $N$ is necessarily odd: 
$$
\{w_n(k):  1\le n\le N\} = \{ 1,  \sin(\omega m k)  ,  \cos(\omega m k)  :  1\le m \le (N-1)/2 \}
$$
An example using a Fourier series is shown in \Cref{f:tri}---details are postponed to \Cref{s:num}.

The \textit{degenerate} family is defined using $N=K$, and
\begin{equation}
w_n(k) = \ind\{ n=k\}\,,\qquad 1\le n,k \le K
\label{e:waveletDegen}
\end{equation}

The optimal control problem with \textit{subspace relaxation} is defined as the optimal control problem 
\begin{subequations} \label{e:primal2}
\begin{align}
J^\star(\nu_0^0) \eqdef  \min_{\nu,\gamma}  \,  \sum_{k=1}^K  & \clD(\nu_k ,\nu_k^0) + \frac{\kappa}{2} \sum_{n=1}^N \gamma_n^2
\label{e:j_star2}
\\
 \text{s.t.} \quad
\gamma_n   &= \lgl \bfmp, \haoutpt_n \rgl - \har_n\,, \quad 1\le n\le N 
\label{e:gamma_constraint2}
\\
 \sum_{u'} \nu_k (\st', u')  &=  \sum_{\st,u}\nu_{k-1} (\st, u)  T_{k-1}(x, \st')  \,, \qquad \ 1\le k\le  K\,, \ s'\in\S
\label{e:dynamics_constraint2}
\end{align}
This reduces to \eqref{e:primal} in the degenerate case \eqref{e:waveletDegen}.

\label{e:primalAndConstraints}
\end{subequations}

The theory that follows is based in part on a relaxation of the dynamical constraints \eqref{e:dynamics_constraint2}, through the introduction of a Lagrange multiplier for each $k$.   This is precisely the first step in the construction of the Hamiltonian in the Minimum Principle approach to optimal control  \cite{lue97}.

\subsection{Duality}
\label{s:klqDual}

Structure for the solution of \eqref{e:primal2} will be obtained by consideration of a dual, in which $\lambda \in \Re^N$ and $g \in \Re^{K \times J^\star}$ denote the vectors of Lagrange multipliers for the first and second set of constraints, respectively.
The matrix $g$ is interpreted as a sequence of functions $g_k\colon\S\to\Re$ that are entirely analogous   to the co-state variables in the Minimum Principle (the Lagrange multipliers for the dynamical constraints) \cite{lue97}.

The Lagrangian is denoted
\begin{align}
 \clL(\nu, \gamma, \lambda, g) =&
\sum_{k=1}^K \clD(\nu_k ,\nu_k^0) + \frac{\kappa}{2} \sum_{n=1}^N \gamma_n^2 + \sum_{n=1}^N \lambda_n  \Bigl( \gamma_n + \sum_{k=1}^K w_n(k) \bigl[r_k - \langle \nu_k, \outpt \rangle \bigr] \Bigr)
\nonumber
\\
& +
\sum_{k=1}^K \sum_{\st'} \Big(\sum_{u'} \nu_k (\st', u') - \sum_{\st,u}\nu_{k-1} (\st, u) T_{k-1}(x, \st')\Big)g_k(\st')
\label{e:Lagrangian}
\end{align}
and the dual function  is defined to be its minimum:
$$
\varphi^\star(\lambda, g) \eqdef \min_{\nu, \gamma} \clL(\nu, \gamma, \lambda, g)
$$
The \textit{dual} of the optimization problem \eqref{e:primal2} is defined as the maximum of the dual function $\varphi^\star$ over $\lambda$ and $g$  (see  \cite{lue97} for a complete and accessible treatment of this theory).
We will see that there is no duality gap,  so that for a quadruple $(\nu^\star,\gamma^\star, \lambda^\star, g^\star)$,    
$$
J^\star(\nu_0^0) = \clL(\nu^\star,\gamma^\star, \lambda^\star, g^\star) =\varphi^\star(\lambda^\star, g^\star) \, .
$$

In the following subsections a representation of the dual function is obtained that is suitable for optimization, which results in a valuable representation for the optimal policy. Properties of the dual function are contained in \Cref{t:dualFunctional} and \Cref{p:algo} that follow. The statement of these results requires additional notation: define a function $\clT^\lambda_k \colon \Re^{|\S|} \to \Re^{|\S|}$, for $f \colon \S \to\Re$ and $\lambda \in \Re^N$, via
\begin{equation}
\begin{aligned} 
\clT^\lambda_k(f;\st) &= \log\Bigl( \sum_u \fee_k^0(u\mid \st) \exp\bigl(\sum_{\st'}T_k(x, \st')  f(\st') + \chlambda_k \outpt(\st,u) \bigr) \Bigr) \,, \quad \st \in \S \, ,
\\
&
\textit{where}
\quad
\chlambda_k = \sum_{n=1}^N \lambda_n w_n(k) 
\label{e:chlambda}
\end{aligned} 
\end{equation}
The maximum of the dual function over $g$ is denoted
\begin{equation*}
\varphi^\star(\lambda) \eqdef \max_g \fee^\star(\lambda,g)=\varphi^\star(\lambda,g^\lambda)
\end{equation*}
where $g^\lambda$ is a maximizer,    $\displaystyle 
g^\lambda \in \argmax_g \fee^\star(\lambda,g)$.
It is shown in \Cref{p:algo} that the vector valued  function $g^\lambda$ satisfies the recursion
\begin{equation}
g_k^\lambda = \clT^\lambda_k(g^\lambda_{k+1}) \, , \quad 1 \le k \le K \,, \quad \text{where} \quad g^\lambda_{K+1} \equiv 0  \,.
\label{e:g_lambda2}
\end{equation}
This forms part of the proof of \Cref{t:dualFunctional}, with complete details postponed to  \Cref{app:B}.
\begin{theorem}
\label{t:dualFunctional}
There exists a maximizer $\{ \lambda_n^\star, g^\star_k : 1 \le n \le N, 1 \le k \le K \}$ for $\varphi^\star$,   
and there is no duality gap:
$$
\varphi^\star(\lambda^\star,g^\star) = J^\star(\nu_0^0) 
$$
The optimal policy is obtained from $\{g^\star_k\}$ via: 
\begin{equation}
\begin{aligned}
\fee_k^\star(u\mid \st) 
&=
\fee_k^0(u\mid \st) \exp\bigl(\sum_{\st'}T_k(x, \st')  g_{k+1}^\star(\st') + \chlambda^\star_k \outpt(\st,u) - g_k^\star(\st) \bigr)\
\\
\text{\it where} \;\; 
g_k^\star (\st) &= 
\clT^\lambda_k(g^\star_{k+1};\st) \;\; \text{\it for} \;\; 1 \le k \le K, \;\; \text{\it and} \;\; g^\star_{K+1} \equiv 0 \,,
\end{aligned} 
\label{e:phi-h-Lambda}
\end{equation}
and  $\{ \chlambda^\star_k \}$ are obtained from $\{ \lambda^\star_n \}$ via \eqref{e:chlambda}.  \qed
\end{theorem}

Denote for each $k$, 
\begin{equation}
\Tg^\lambda_k(x) = \sum_{\st}T_{k-1} (x, \st)g^\lambda_{k}(\st)
\label{e:Gamma}
\end{equation}
The proof of the following  is also contained in \Cref{app:B}.
\begin{proposition}
\label[proposition]{p:algo}
The following hold for the dual of  \eqref{e:primal2}: for each $\lambda\in\Re^N$,
\begin{romannum}
\item A maximizer $g^\lambda$ is given by \eqref{e:g_lambda2}

\item The maximum of the dual function over $g$ is the concave function
\begin{equation}
\varphi^\star(\lambda) = \lambda^T \har
-\frac{1}{2\kappa} \|\lambda\|^2 - \lgl \nu_0^0, \Tg^\lambda_1 \rgl
\label{e:dualAlambda}
\end{equation}

\item The function \eqref{e:dualAlambda} is continuously differentiable, and
\begin{equation}
\frac{\partial}{\partial\lambda_n}  \varphi^\star(\lambda)  =  \har_n - \frac{1}{\kappa}  \lambda_n - \sum_{k=1}^{K} w_n(k) \lgl \nu^\lambda_k, \outpt \rgl \,, \qquad 1\le n\le N
\label{e:phi_der_gen}
\end{equation}
where $\{\nu^\lambda_k\}$ is the sequence of marginals obtained from the randomized policy defined in \eqref{e:phi-h-Lambda}, substituting $\{g^\star_k\}$ by $\{g^\lambda_k\}$ defined in (i). 
\qed
\end{romannum}
\end{proposition}

To conclude this section, we provide representations of the log-likelihood ratio, $L(\vx)$, relative entropy $D(\bfmp^\lambda \| \bfmp^0)$, and primal objective function for the pmf $\bfmp^\lambda \in \clS(\state^{K+1})$ obtained from the randomized policy defined in \eqref{e:phi-h-Lambda}, substituting $\{g^\star_k\}$ by $\{g^\lambda_k\}$ defined in \Cref{p:algo}, part (i). The proof of the following  is contained in \Cref{app:B}: 
\begin{corollary}
\label[corollary]{c:optLLR}
The following hold for all $\{ \chlambda_k, g^\lambda_k : 1 \le k \le K \}$:
\begin{romannum}
\item The log-likelihood ratio can be expressed:
\begin{equation}
L(\vx) =  \sum_{k=1}^{K} \{ \Delta_{k}(x_{k-1},s_{k}) + \chlambda_k \outpt(x_k)\} - \Tg^\lambda_1(x_0)
\label{e:llr}
\end{equation}
where for each $k$  \emph{(recalling $x_k=(\st_k, u_k)$)},      
\begin{equation}
\Delta_{k}(x_{k-1},\st_{k}) =   \Tg^\lambda_k(x_{k-1})   - g^\lambda_{k}(\st_{k})
\label{e:delta}
\end{equation}
\item The relative entropy is given by
\begin{equation}
D(\bfmp^\lambda \| \bfmp^0) = \sum_{k=1}^{K} \chlambda_k \lgl \nu^\lambda_k, \outpt \rgl 
				- \lgl \nu_0^0, \Tg^\lambda_1 \rgl
\label{e:rel_ent}
\end{equation}

\item The value  of the primal is given by
\begin{subequations}
\begin{align}
J(\bfmp^\lambda, \nu_0^0) &\eqdef D(\bfmp^\lambda \| \bfmp^0) + \frac{\kappa}{2} \sum_{n=1}^N \gamma_n^2
\label{e:primal3}
\\
&= - \lgl \nu_0^0, \Tg^\lambda_1 \rgl + \sum_{k=1}^{K} \chlambda_k \lgl \nu^\lambda_k, \outpt \rgl  + \frac{\kappa}{2} \sum_{n=1}^N \gamma_n^2 
\label{e:primal4}
\end{align}
with $\gamma_n =  \lgl \bfmp^\lambda, \haoutpt_n \rgl - \har_n$.
\end{subequations}
\qed
\end{romannum}
\end{corollary}
The stochastic process $\{\Delta_{k}(X_{k-1},S_{k}) \}$ is a martingale difference sequence; it vanishes when nature is deterministic,  reducing to the solution obtained in \cite{cambusjimey19}.
 
\begin{wrapfigure}[13]{R}[-2pt]{0.325\textwidth}

\centering

  	\includegraphics[width=.95\hsize]{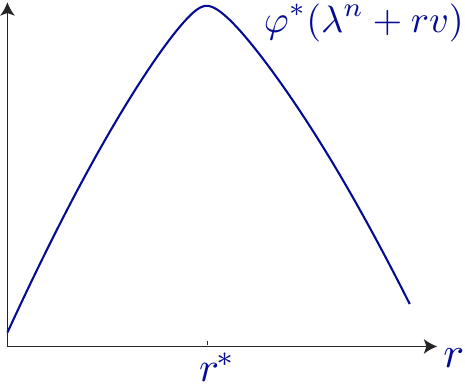}
	
	\caption{\small
	\small	 Dual function along a line-segment}
	\label{f:phi_plot.pdf}
\end{wrapfigure}

\subsection{Algorithms} \label{s:alg}
Given the simple form of the derivative \eqref{e:phi_der_gen}, it is tempting to apply gradient ascent to obtain $\lambda^\star$. The difficulty with standard first-order methods is illustrated in \Cref{f:phi_plot.pdf}. This is a plot of a typical example in which $\lambda^n\in\Re^N$ is given, $v= \nabla \varphi^\star\, (\lambda^n)$, and the plot shows $\varphi^\star\, (\lambda^n +r v)$ for a range of positive $r$. We have found in examples that using gradient ascent on this cone-shaped curve may be slow to converge, likely due to a large ``overshoot'' when applying standard first-order methods.

In the numerical results that follow we opt for proximal gradient methods \cite{parboy13}.

\wham{Monte Carlo methods.}
The gradient of the dual function may be expressed in terms of the first-order statistics of the random variables $\{\haoutpt_n(\vX): 1\le n\le N\}$ when $\vX\sim \bfmp^\lambda$:
\begin{equation}
\begin{aligned}
\Expect [\haoutpt_n(\vX)]
&= \sum_{\vx} \bfmp^\lambda(\vx) \sum_{k=1}^{K} w_n(k) \outpt(x_k)
\\ &= \sum_{k=1}^{K} w_n(k) \sum_{x_k} \sum_{x_i, i \neq k} \bfmp^\lambda(\vx) \outpt(x_k)
 = \sum_{k=1}^{K} w_n(k) \lgl \nu^\lambda_k, \outpt \rgl
\label{e:Moments}
\end{aligned}
\end{equation}

\Cref{t:phi_der} follows from  \eqref{e:phi_der_gen} combined with \eqref{e:Moments}:
\begin{lemma}
	\label{t:phi_der}
	For any $\lambda\in\Re^N$ and $1 \le n \le N$,
	\begin{equation}
	\frac{\partial}{\partial\lambda_n}    \varphi^\star(\lambda) = \har_n - \frac{1}{\kappa}  \lambda_n - \Expect[\haoutpt_n(\vX)]\,,
	\quad
	\textit{in which $\vX \sim \bfmp^\lambda$.}
	\end{equation}
	\qed
\end{lemma}

See \cite{busmeycam23} for more on Monte Carlo methods and KLQ.

\section{Feedback Formulations}
\label{s:IPD-Q}

We now turn  to the  IPD-Q convex program \eqref{e:IPDQ}.     
  It is   assumed throughout this section that $T_k =T$ and  $\fee_k^0= \fee^0 $, independent of $k$.

 The relationship between IPD-Q and \eqref{e:KLQinfty} will be clear after justification of the term   $\clDV(\spmf,\fee)  $ defined in  \eqref{e:KLrateSS}.
 Consider any $\fee\in\Upphi$, which gives rise to a Markov chain with transition matrix $P_\fee$.
The relative entropy \eqref{e:KL} was previously expressed as a sum over  $\vx\in\state^{K+1}$ in 
\eqref{e:KL}.   The notation $D(\bfmp\| \bfmp^0) =D^K(\bfmp\| \bfmp^0) $ is required in the following, since $K$ is a variable in  \eqref{e:KLQinfty}.

\begin{proposition}
\label[proposition]{t:DV}
Suppose that $\bfmp$ is obtained using the policy $\fee$, and  initial pmf $\nu_0^0$ common with $\bfmp^0$.   
Suppose moreover that $P_\fee$ has a unique invariant pmf $\pi$.
Then,
\[
\clDV(\spmf,\fee)  
=
\lim_{K\to\infty}  \frac{1}{K}     \sum_{k=1}^{K}  \clD(\nu_k ,\nu_k^0)
=
\lim_{K\to\infty}  \frac{1}{K} D^K(\bfmp\| \bfmp^0)   =  \DVrate(P_\fee\| P^0) 
\]
where $\DVrate$ denotes the rate function \eqref{e:DVrate} using $P=P_\fee$:
\begin{equation*}
\DVrate(P_\fee\| P^0) = \sum_{x,x'} \pi(x) P_\fee(x,x')   \log \Bigl(\frac{P_\fee(x,x') }{P^0(x,x')} \Bigr)
\end{equation*}  
\end{proposition}

\begin{proof}
The proof of the first identity begins with
\[
 \frac{1}{K}     \sum_{k=1}^{K}  \clD(\nu_k ,\nu_k^0)  =   \frac{1}{K}     \sum_{k=1}^{K}  \Expect[F(X_k))] 
\]
with $F(x) = \log[{\fee_k(u\mid \st)}/{\fee_k^0(u\mid \st)} ]$ for $x=(\st,u)\in\state$.     The average converges to $\clDV(\spmf,\fee) $ as $K\to\infty$ since the invariant pmf $\pi$ is unique.   
\end{proof}

The distinct approaches to optimal control pursued in this section follow the distinct approaches to optimal control in general,  via the HJB equations and optimal control via the Minimum Principle (MP):
\wham{(i)}
In \Cref{s:IPDQ_HJB} IPD-Q is interpreted as a solution to an HJB equation, which results in a solution in state feedback form,    $\fee_k^\star = \clK^\star (\nu_k^\star, r)$,  for some mapping $\clK^\star \colon \clS(\state)\times\Re \to\Upphi$.   The solution to IPD-Q is  $\fee^\star(u\mid \st) = \clK^\star (\nu^r, r)$,  in which $\nu^r$ is the steady-state marginal for $\bfmS$ under the IPD-Q policy.

  Computation of $\clK^\star$ may be difficult if   the state space is large.   An LQR approximation is proposed,   justified for small $|r|$,  and the approximation  \eqref{e:SmallSignalPolicy} may also be found at the close of \Cref{s:IPDQ_HJB}.

 \wham{(ii)}  The approach taken in \Cref{s:ACOE-IPDQ}   is in essence the infinite-$K$ limit of the approach taken in  \Cref{s:klqDual}  which,
  as noted following \eqref{e:primalAndConstraints},
   is the  Minimum Principle (MP) approach.   It is well known that this approach provides only an open-loop solution.

\subsection{HJB approach}   
\label{s:IPDQ_HJB}
 
The solution to the optimal control problem \eqref{e:IPDQ-global} may be characterized using techniques from deterministic optimal control theory.     

The ACOE holds for deterministic systems,  precisely as reviewed in \Cref{s:LQRmot} for the linear quadratic problem:
\begin{equation}
 \eta^\star + \clH^\star(\nu) = \min_\fee \bigl\{    c(\nu,\fee;r)  +   \clH^\star(   f(\nu,\fee) ) \bigr\}\,,\qquad \nu\in \clS(\state)
\label{e:KLQACOE}
\end{equation}
with $c(\nu,\fee;r)  $ defined in  \eqref{e:IPDQ-global},
$\clH^\star \colon \clS(\state) \to\Re$ the relative value function, and $ \eta^\star$ the optimal average cost.
 The minimizer $\fee^\star$ defines   $\fee_k^\star = \clK^\star (\nu_k^\star, r)$.

We are not aware of solution techniques for this instance of the ACOE, beyond the standard value iteration algorithm or other generic approaches.   

The relative value function $\clH^\star$ and feedback law  $\clK^\star $ can be approximated through  a small signal linearization of the dynamics, and a quadratic approximation of the cost.    
We begin with an approximation for the latter.  

The proof of \Cref{t:IPDQ-LQRa}  follows from the definition \eqref{e:KLrateSS} and 
a Taylor's series approximation of the logarithm.    For any $\fee$, denote  by $\tilfee(u\mid \st) \eqdef\fee(u \mid \st) - \fee^0(u \mid \st) $ the   deviation.
\begin{proposition}
\label[proposition]{t:IPDQ-LQRa} 
Suppose that $  \lgl \pi^0, \outpt \rgl =0$.   Then, 
the cost function \eqref{e:IPDQ-global} is nearly quadratic in deviations:
\[ 
  c(\nu,\fee;r)  =     \| \tilfee \|_R^2   +   \frac{\kappa}{2} (y-r)^2  +O( \| \tilfee \|_R^3)
\]
in which $y = \sum_x [\nu(x) - \pi^0(x)]\outpt(x)$,   and
$\displaystyle
 \| \tilfee \|_R^2  =  \sum_{\st,u}    \frac{\spmf(\st)} { \fee^0(u\mid \st)  }  \tilfee(u\mid\st)^2
$.
\end{proposition}

Approximation of the dynamics by a linear system is justified when $|r|$ is small,   and $\nu_0^0\approx \pi^0$,  the invariant pmf for $P^0$.    
 The corresponding stationary pmf for $\bfmS$ is denoted $\spmf^0$  (recall \eqref{e:pimu}).

Let $\state= \{x^i :  1\le i\le n\} $   with $n=|\state|$.  The LQR approximation has state denoted $\tilclX_k$ and input $\tilclU_k$ at time $k$,  with   $\tilclX_k^i$   an approximation of $\tilnu_k(x^i) \eqdef \nu_k(x^i)   -\pi^0(x^i) $,  and $\tilclU_k^i$ an approximation of $\tilfee_k^i(u^i \mid \st^i)  \eqdef\fee_k^i(u^i \mid \st^i) - \fee^0(u^i \mid \st^i) $.    The definition of the linearization is a system model of the form \eqref{e:LQR},
 \[
 \tilclX_{k+1}  =  A  \tilclX_k  +   B   \tilclU_k\,,   \qquad   \tilclY_k = C^\transpose  \tilclX_k  
 \]
 in which $ \tilclY_k $ is an approximation of $\lgl  \tilnu_k, \outpt\rgl$.
 Expressions for the $n\times n$ matrices $A$ and $B$,  and the $n$-dimensional column vector $C$,  are provided in the following.

\begin{proposition}
\label[proposition]{t:IPDQ-LQR}
The small signal approximation holds with
\[
A_{i,j}  = P^0(x^j,x^i) \,,\quad  B_{i,j} =  \ind\{i=j\}  \spmf^0(\st^j)
 \,,\quad 
 C_i = \outpt(x^i)
 	\,,\qquad 1\le i,j\le n
\]
\end{proposition}

\begin{proof}
The expression for $C$ is by definition of $\tilclY_k$.   The other matrices are obtained through the standard first-order Taylor series approximations:
 \[
 A_{i,j}   \eqdef  \frac{\partial}{\partial \nu^j }  f_i(\nu,\fee)  \Big|_{\nu=\pi^0,\fee=\fee^0}
  =  P^0(x^j,x^i)    
\]
with $\nu^j= \nu(x^j)$.

The input $\tilclU_k$ is an $n$-dimensional column vector, so that $B$ is an $n\times n$ matrix.   It is obtained from the Taylor series approximation,
 \[
 B_{i,j}   \eqdef  \frac{\partial}{\partial \fee^j }  f_i(\nu,\fee)  \Big|_{\nu=\pi^0,\fee=\fee^0}
  =   \ind\{i=j\} \sum_{x\in\state} \pi^0(x) T(x,\st^j) 
\]
where  $ \fee^j = \fee(u^j \mid \st^j)$.
By invariance of $\pi^0$ it follows that $B$ is diagonal, with $i$th diagonal entry equal to $\spmf^0(\st^i)$.
\end{proof}
 
\Cref{t:IPDQ-LQRa,t:IPDQ-LQR} imply that for small $r$,  the solution to the nonlinear optimal control  problem is approximated by the average-cost LQR  solution using
 \[
  c(\tilclX,\tilclU;r)  =     \| \tilclU \|_R^2   +   \frac{\kappa}{2} (\tilclY-r)^2  \,,  \quad \tilclY = C^\transpose \tilclX
 \]
 giving $\tilclU_k = - K^\star \tilclX_k  +  G^\star  r$,  with gain matrices  $ K^\star $   ($n\times n$) and $ G^\star $   ($n\times 1$). 
 
 This leads to the policy approximation.   Write $\tilclU_k = \clU_k - \clU_k^0$  with  $\clU_k^0$ the  vector representation of the nominal policy.   The $i$th entry of the input is expressed
 \[
 \begin{aligned} 
 \clU_k^i  &=   \fee^0(u^i\mid \st^i)  +[- K^\star \tilclX_k  +  G^\star  r]_i
 \\
& =  \fee^0(u^i\mid \st^i)  \Bigl[  1 +  \frac{1}{\fee^0(u^i\mid \st^i)}  \Bigl(   - [K^\star \tilclX_k]_i  +  G^\star_i  r  \Bigr)  \Bigr]
 \end{aligned} 
 \]
This implies the small signal approximation \eqref{e:SmallSignalPolicy}.
It is conjectured that  \eqref{e:SmallSignalPolicy} is within $O(r^2)$ of optimal (in terms of the objective in \eqref{e:IPDQ}).

\subsection{Minimum Principle approach}
\label{s:ACOE-IPDQ}

As previously observed,
the optimization problem \eqref{e:IPDQ} falls outside of traditional MDP theory:
\begin{romannum}
\item The control cost  is absent, and is replaced by a cost on the randomized policy. 
\item 
A quadratic cost on $\pi$ appears, rather than linear as anticipated in the LP formulations of MDPs.   
\end{romannum}
An MDP model is constructed here through a series of steps, with the first step addressing (i).  
For this it is natural  to view the input as an element of the simplex $\clS(\U)$.
This is not the same setting as \Cref{s:IPDQ_HJB}:   in this subsection,
the notation $ \fee(\varble \mid \st)$ is interpreted as static state feedback from state $\st$ to input $ \fee(\varble \mid \st)  \in \clS(\U)$.

To remove the quadratic cost on $\pi$ requires a Lagrangian relaxation, similar to what was used in \Cref{s:klq}.  For $\lambda\in\Re$ denote
\begin{equation}
[\pi^{\lambda} ,\fee^{\lambda} ,\gamma^\lambda] =\argmin_{\pi,\fee,\gamma  }  \bigl\{  \clD(\fee )   
				+ \frac{\kappa}{2}   \gamma^2   + \lambda [ \gamma - \lgl \pi, \outpt \rgl + r ]     :   \pi P_\fee = \pi  \bigr\}
\label{e:IPDQrelax}
\end{equation}
For each $\lambda$ this is viewed as a standard average cost optimal control problem with state process $\bfmS$.
The controlled transition matrix  and cost function are defined by
\[
\begin{aligned}
T_\upmu(\st,\st')   &\eqdef   \sum_u \upmu(u) T((u,\st), \st')\,,\qquad    &&\st,\st'\in\S\, ,\quad  \upmu \in \clS(\U) \,,
\\
c(s, \upmu)  &\eqdef    
\sum_u \upmu(u) \Bigl[     \log\Bigl(\frac{\upmu(u)}{\fee^0(u\mid \st)}  \Bigr) -\lambda \outpt(\st,u)   \Bigr] &&\st\in\S\, ,\quad  \upmu \in \clS(\U) \, .
\end{aligned} 
\]

Under any policy $\fee \in \Upphi$ the resulting process $\bfmS$ is Markovian.   With a slight abuse of notation, its transition matrix is denoted 
\[
T_\fee(\st,\st')   \eqdef   \sum_u \fee(u\mid \st) T((u,\st), \st')\
\]
and the cost as a function of $\st$ under this policy is denoted   
\[
 c_\fee(\st)  =   \sum_u \fee(u\mid \st) c(s, \fee(u \mid \st) ) =   \sum_u \fee(u\mid \st) 
\Bigl[     \log\Bigl(\frac{\fee(u\mid \st)}{\fee^0(u\mid \st)}  \Bigr) -\lambda \outpt(\st,u)   \Bigr] \,, \quad \st\in\S\, .
\]
The solution to \eqref{e:IPDQrelax} gives $\gamma^\lambda= -\lambda/\kappa$ and
\begin{equation}
[\pi^{\lambda} ,\fee^{\lambda} ] =\argmin_{\spmf,\fee  }  \Bigl\{    \sum_\st  \spmf(\st) c_\fee(\st)     :   \spmf T_\fee = \spmf  \Bigr\}
\label{e:IPDQrelaxB}
\end{equation}
This is a standard MDP formulation, in which the optimization over feedback laws $\fee$ is explicit.

The Lagrange multiplier $\lambda$ is treated as the independent parameter rather than $r$.   This is justified through the correspondence $\gamma^\lambda= -\lambda/\kappa$, and the following definition imposes complementary slackness  
\begin{equation}
r^\lambda=  -\gamma +  \lgl \pi^\lambda, \outpt \rgl   =    \lgl \pi^\lambda, \outpt \rgl +  \lambda/\kappa
\label{e:rCS}
\end{equation}
As $\lambda$ ranges from $-\infty$ to $+\infty$,   so do the values of $r^\lambda$ because $  \lgl \pi^\lambda, \outpt \rgl $ is bounded and continuous in $\lambda$.    

Continuity of  $  \lgl \pi^\lambda, \outpt \rgl $ and other conclusions are obtained from prior research (in particular   \cite{busmey18a}),  because the   optimization problem \eqref{e:IPDQrelaxB} is identical to the IPD optimization problem \eqref{e:IPD}, in which $\zeta$ is replaced by $\lambda$.  

To match the setting of \cite{busmey18a}, denote the one-step reward as the negative of cost, 
$\varrho(s,\fee) =-c(s,\fee)     $.   
Based on the foregoing, the solution to \eqref{e:IPDQrelaxB}
 is characterized by the average reward optimality equation
\begin{equation}
\xi^\lambda +  h^\lambda(\st) = \max_\fee\bigl \{ \varrho(\st,\fee)  +  \sum_{\st'}  T_\fee(\st,\st') h^\lambda(\st') \bigr \}
\label{e:ACOE-IPDQ}
\end{equation}
The maximizer provides a representation for the optimal policy similar to \eqref{e:phistar1}:
\begin{equation}
\fee^\lambda(u\mid \st) 
=
\fee^0(u\mid \st) \exp\bigl( \barh^\lambda (\st,u)  + \lambda \outpt(\st,u)- \Gamma^\lambda(\st) \bigr)\,,
\label{e:phistarIPDQ}
\end{equation}
with $\barh^\lambda(x) = \sum_{u'} T(x , \st')  h^\lambda(\st')$ and $\Gamma^\lambda(\st) $ the normalizing factor,
\begin{equation}
\Gamma^\lambda(\st) = \log \sum_u  \fee^0(u\mid \st) \exp\bigl( \barh^\lambda (\st,u)  + \lambda \outpt(\st,u) \bigr)\,.
\label{e:Gamma-lambda}
\end{equation}

\wham{ODE solution}   The reader is referred to \cite{busmey18a} for full details on this solution technique to compute the solution to \eqref{e:ACOE-IPDQ}.    The main ideas are recalled here, in part because they are required in a small signal approximation.   

It is shown in this prior work that the relative value functions can be constructed so that they are continuously differentiable in $\lambda$.   Letting $H^\lambda =  \frac{d}{d\lambda } h^\lambda$, the following is obtained:
\[
\baroutpt^\lambda
+  H^\lambda(\st) =    \outpt_\st  +  \sum_{\st'}  T_\lambda(\st,\st') H^\lambda(\st')    \,,
\quad \st\in \S\,, \   \lambda\in\Re\,,
\]
in which  $   T_\lambda =  T_{\fee^\lambda}$,
\begin{equation}
 \outpt_\st \eqdef  \frac{d}{d\lambda }  \varrho(s,\fee) = \sum_u \fee(u\mid \st) \outpt(\st,u)  
 \,
 \quad \textit{and}\quad
 \baroutpt^\lambda = \frac{d}{d\lambda }  \xi^\lambda
\label{e:PoissonElementsY}
\end{equation}
  This fixed point equation is known as Poisson's equation,  whose solution is often expressed $H^\lambda = Z_\lambda \outpt$ with $Z_\lambda$ known as the fundamental matrix  (obtained as a simple matrix inverse).
Also obtained is  
\begin{equation}
\baroutpt^\lambda  = \sum_x  \pi^\lambda(x)   \outpt(x)
\label{e:barY}
\end{equation}
where $\pi^\lambda (\st,u) = \spmf^\lambda(\st) \fee^\lambda(u\mid \st)$,  with $\spmf^\lambda$  the unique invariant pmf  for $T_{\lambda}$.

This defines the ODE solution for the family of relative value functions
\begin{equation}
\frac{d}{d\lambda} h_\lambda =  Z_\lambda \outpt  
\label{e:ODEsoln}
\end{equation}
with boundary condition $h^\lambda\equiv 0$ when $\lambda=0$.   The right hand side depends on $h_\lambda$ through $  Z_\lambda$,  but the dependency is smooth.

\wham{Small signal approximation}

The small signal approximation here is defined in a setting similar to 
\Cref{t:IPDQ-LQR}:  it is assumed that the reference signal is small in magnitude, and  that $r\equiv 0$ achieves zero cost in  \eqref{e:IPDQ}.   This holds if $\sum\pi^0(x) \outpt(x) =0$, which will be assumed henceforth.

A slight change in notation is required here, as compared to \Cref{s:IPDQ_HJB}:
 $ \tilX_k  $  and $ \tilU_k$ are $n$-dimensional column vectors that denote the exact deviation:
$ \tilX_k^i \eqdef  \tilnu_k(x^i)$   and  $ \tilU_k^i =\tilfee_k(u^i \mid \st^i)$ for each $i$ and $k$.  
The approximation requires the following notation:
\begin{romannum}

\item
$\displaystyle
\varsigma^2_\lambda=  \frac{d^2}{d\lambda^2} \xi^\lambda $ for $ \lambda\in\Re$.

\item  $\barH^\lambda = \frac{d}{d\lambda} \barh^\lambda$,   
\
\
   $\displaystyle 
\barH^\lambda_\st =  \sum_u  \fee^0(u\mid \st) \barH^\lambda(\st,u)$,    
\
\
$\displaystyle 
\outpt_\st =  \sum_u  \fee^0(u\mid \st) \outpt(\st,u)$.

\item $\Uplambda(x) =\barH^0(x) + \outpt(x)  -   ( \barH^0_\st + \outpt_\st) $,
\
\
 $\UplambdaN (x) = \bigl(  \varsigma^2_0 + 1/ \kappa \bigr)^{-1} \Uplambda (x) $. 
 \end{romannum}

Approximation of the state dynamics begins with an approximation of the input.   The proof of \Cref{t:LQRbLemma} is postponed to \Cref{app:IPDQ}.   
\begin{lemma}
\label[lemma]{t:LQRbLemma}
The small-$r$ approximation holds for the solution to IPD-Q:
 \begin{equation}
  \tilfee (u\mid s,r)  =    \fee^0 (u\mid \st)  \exp\bigl(  \UplambdaN(x) r  \bigr) + O(r^2)
\label{e:tilfee}
\end{equation}
\qed
\end{lemma}

The following linear systems approximation follows easily from  \Cref{t:LQRbLemma}.

\begin{proposition}
\label[proposition]{t:IPDQ-LQRb} 
Suppose that the input  $\fee_k(u\mid \st) = \fee^\star(u\mid \st, r_k)$ is applied to the nonlinear system \eqref{e:IPDQ-global}.   The closed loop dynamics then admit   the approximation 
 \begin{equation}
\tilX_{k+1}  =  A \tilX_k  +   B  G^\star r_k       +  \epsy_k + O(r_k^2)
\label{e:mfmXr}
\end{equation}
in which $A$ and $B$ are defined in \Cref{t:IPDQ-LQR},  
$ G^\star$ is the column vector   with entries  $ G^\star_i  =  \fee^0 (u^i\mid \st^i)  \UplambdaN(x^i)$, 
and $\epsy_k$ is quadratic in the deviation $(\tilnu_k , \tilfee_k)$:
\[
 \epsy_k^i  = \sum_j  \tilnu_k(x^j)  T(x^j, \st^i) \tilfee_k(u^i \mid \st^i)\,,\qquad x^i = (\st^i,u^i)\in\state
\]
\qed
\end{proposition}

\section{Applications to Demand Dispatch}
\label{s:num}

An application of the control framework described in the previous sections is \textit{Demand Dispatch}: an evolving science for automatically controlling flexible loads to help maintain supply-demand balance in the power grid. The goal of demand dispatch (DD) is to modify the behavior of flexible loads such that the aggregate power consumption tracks a reference signal that is broadcast by a balancing authority (BA).

Keep in mind that in the numerical examples here we focus entirely on the mean-field model.     We know from prior work that evolution of the empirical distributions does closely track this idealization:  for reasonably large $\clN$,   following the notation \eqref{e:empDist}, the approximation $\nu^{\, \clN}_k  \approx \nu_k$ holds and the covariance of the error grows slowly with $k$ (error is reduced with feedback \cite{YueChenThesis16,chebusmey17a}).   
Although the   control architecture in this prior work is very different, it should not surprise the reader that the law of large numbers and associated central limit theorem hold in the setting of this paper.

Also, the numerical results here focus entirely on the solutions surveyed in \Cref{s:klq}.   As explained in \Cref{s:IPD-Q}, the
 IPD-Q solution for real-time feedback reduces to something similar to what has been extensively explored in prior work    \cite{YueChenThesis16,chebusmey17a}.


Although these techniques can be applied to any flexible load, the experiments in this section demonstrate distributed control of a population of residential water heaters or refrigerators.   An MDP model is constructed in which the state is the standard used to capture hysteresis control,  $	S_k = (\theta_k, U_{k-1})$,
in which $\theta_k\in\Re$ is the temperature, and
$U_k \in \{0,1\}$ denotes power mode for each $k$.  
Remember the physical system operates in continuous time, and $k$ represents the $k$th sampling time.    This means that  $U_{k-1}$ represents the power mode during the   sampling interval ending at the $k$th sampling time.


\begin{figure}[h]
	\centering
	\includegraphics[width=.6\hsize]{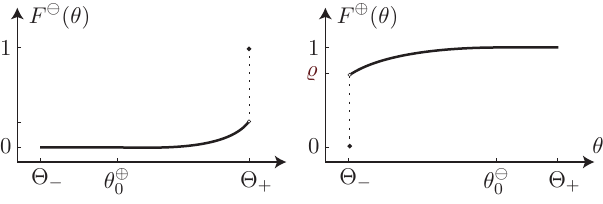}
	
	\caption{\small
	Nominal model for a water heater.} 
	\label{f:WH}
\end{figure}

\subsection{Designing the nominal model}

Construction of the nominal model with transition matrices $\{P^0_k\}$ of the form \eqref{e:nomP} requires specification of dynamics of nature and the nominal policy.     In the case of water heaters, the sequence of transition matrices  $\{T_k\}$ for nature were based on  input-output data obtained from Oak Ridge National Laboratories \cite{chehasmatbusmey18}. 
For refrigerators, $T$ was taken independent of $k$,  constructed based on simulations of the standard linear TCL model:
\begin{equation}
	\theta_{k+1} = \theta_k + \alpha [\theta^a - \theta_k] - \beta U_k + D_{k+1}  \,, 
\label{e:StateTCL_dist}
\end{equation}
in which $\alpha, \beta>0$, $\theta^a$ denotes the (time-invariant) ambient air temperature, and the disturbance process $\bfmD$ captures modeling error and usage.        

In all cases the nominal policy was chosen time-homogeneous: $\fee^0_k \equiv \fee^0$ is a fixed randomized policy, designed to approximate   deterministic hysteresis control.  We describe the construction for water heaters, following \cite{meybarbusyueehr15, chehasmatbusmey18}.   

The upper and lower temperature limits that define a deadband are denoted $\Theta_-$, $\Theta_+$, respectively. A standard residential water heater switches deterministically when it reaches the limits, but $\fee^0$ is constructed so that the power mode will switch stochastically, often before reaching one of the two limits. The randomized decision rule is represented by two CDFs: $F^\oplus$ is the CDF for the temperature at which power turns on, and $F^\ominus$ is the CDF for the temperature at which power turns off.

A particular design for these CDFs, shown in \Cref{f:WH}, is obtained on fixing parameters $\theta_0^\oplus, \theta_0^\ominus \in [\Theta_-,\Theta_+]$, $\varsigma\in(0,1)$ and $\eta > 1$:
\[
\begin{aligned}
F^\ominus(\theta) &= \varsigma(\theta - \theta_0^\ominus)_+^\eta,
\\ 
F^\oplus(\theta) &=  1-  \varsigma(\theta_0^\oplus - \theta)_+^\eta
\,, \qquad \theta \in [\theta_-,\theta_+)  \,.
\end{aligned}
\]
Without loss of generality it is assumed that the sampling interval is 1~unit. At time instance $k$, the decision rule is:
\begin{romannum}
	\item If $U_k = 0$ then $U_{k+1} = 1$ with probability
$\displaystyle
	\fee^0_k(1 | \st) = \frac{[F^\oplus(\theta_{k-1}) - F^\oplus(\theta_k)]_+}{F^\oplus(\theta_{k-1})}
$.
	
	\item If $U_k = 1$ then $U_{k+1} = 0$ with probability
$\displaystyle
	\fee^0_k(0 | \st) = \frac{[F^\ominus(\theta_k) - F^\ominus(\theta_{k-1})]_+}{1-F^\ominus(\theta_{k-1})}
$.
\end{romannum}

\begin{figure}
	\centering 
		\centering
		\includegraphics[width=0.9\hsize]{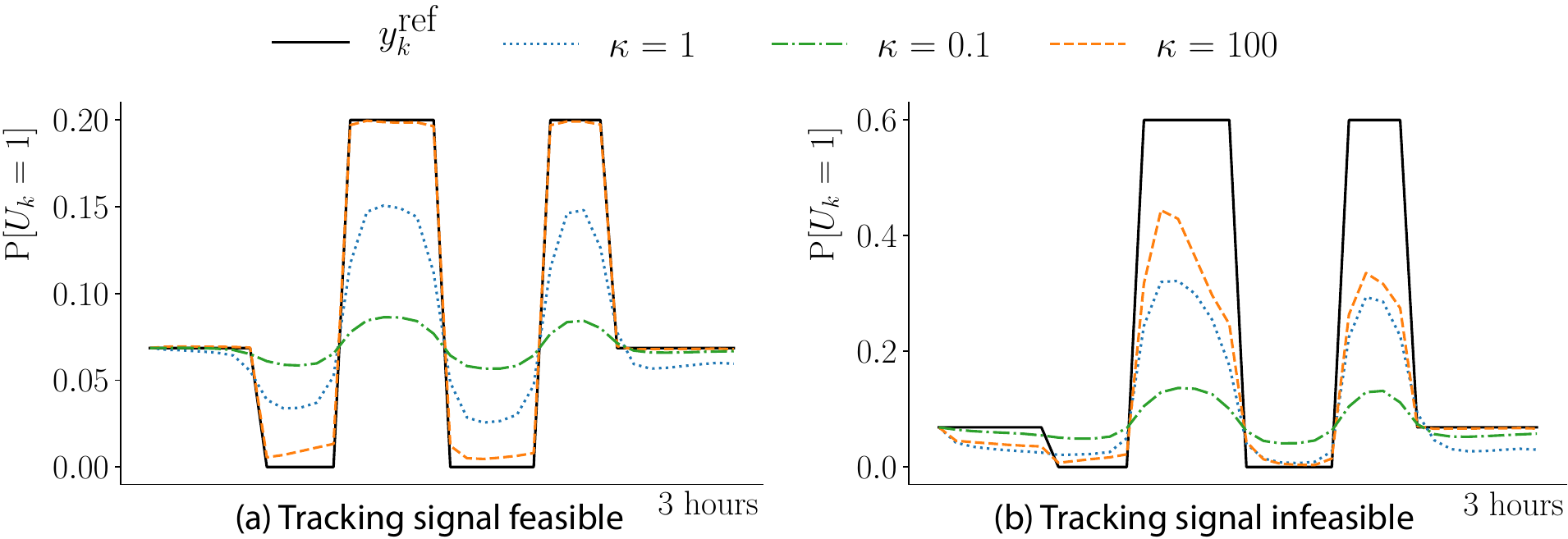}

		\caption{\small
	Evolution of $\clN \langle \nu_k, \outpt \rangle$:  (a)  reference signal is feasible;  (b) reference signal is  infeasible.}
		\label{f:wh_square_both}		
	 
\end{figure}

\subsection{Tracking}

In practical applications the  aggregate power is of interest, which is approximated by $\pwr\clN y_k$ at time $k$,  where $\pwr$ is the rated power of a single load. Hence the total population size $\clN$ must be taken into account in any tracking problem. 
 In plots that follow, we choose to focus on the ``normalized'' response, defined as follows:
\begin{equation*}
		y^{\textrm{ref}}_k = r_k/\pwr \, , \qquad \qquad
		\hat{y}^{\textrm{ref}}_k = \har_k/\pwr \, , \qquad \qquad
		y_k = \langle \nu_k, \outpt \rangle/\pwr 
		\label{e:util}
\end{equation*}
  In this context, $y_k$ can be interpreted as the probability of a load being on.

\begin{wrapfigure}[16]{L}[-1pt]{0.45\textwidth}
		\centering
		\includegraphics[width=0.95\linewidth]{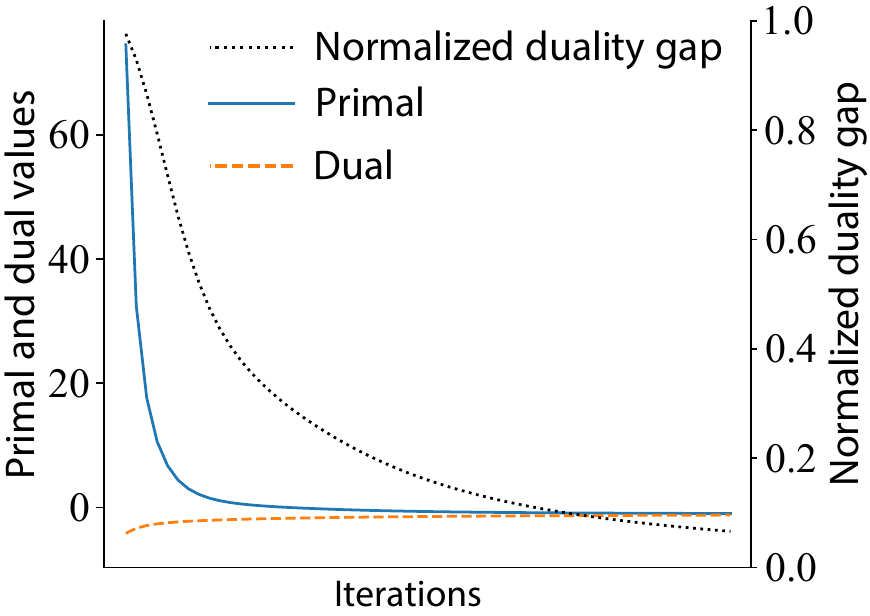}
		\caption{\small
	The normalized duality gap approaches zero in each experiment.}
		\label{f:DG}
\end{wrapfigure}

The two sets of plots in \Cref{f:wh_square_both} are distinguished by the reference signal.   In each case the reference signal is a square wave.   In (a) the signal is feasible, and in (b) it violates the energy limits of the collection of water heaters \cite{haosanpoovin15}. In \Cref{f:wh_square_both}~(a) it is seen that tracking is nearly perfect for sufficiently large $\kappa$.   Tracking of the larger reference  signal would require temperature deviations to exceed the deadband of the water heater. Instead, we observe in \Cref{f:wh_square_both}~(b) a graceful truncation of the reference signal.

The next experiment utilizes a subspace relaxation, via a Fourier transformation, to demonstrate the potential to reduce computational complexity. One weakness of this approach is the introduction of undesirable oscillations in the transformed reference signal, as shown in \Cref{f:tri}. The authors are currently investigating alternative transform techniques. \Cref{f:DG} displays a result typical of all of our experiments: the normalized duality gap\footnote{the duality gap divided by the value of the primal} 
tends towards zero.

\Cref{f:hist-evo} displays the results of a tracking experiment comparing six different initial conditions. Observe how their marginal distributions become nearly identical within a few hours. Recall that this control problem requires knowledge of the initial distribution $\nu_0$. These results suggest that an accurate estimate of the global marginal distribution can be readily available at each load. This has interesting implications for control design;  see \Cref{s:con_arc} for further discussion.

\subsection{Sensitivity}
The next set of experiments was designed to assess sensitivity of KLQ optimal control to modeling error. 
Specifically, what are the consequences of ignoring the randomness of nature?

A particular choice of statistics for \eqref{e:StateTCL_dist} was chosen in order to mimic the effect of a refrigerator door opening at random times throughout the day:  $\bfmD$ is i.i.d., with
\begin{equation*}
	D_{k+1} = \begin{cases} 
		\bar{d} & \text{with probability} \quad \epsy  \\
		0 & \text{with probability} \quad 1 - \epsy
	\end{cases}
\end{equation*}
where $\epsy$ determines the average amount of door openings per day, and $\bar{d}$ was chosen so that the temperature inside the refrigerator increases when the door is open even when the power mode is on.   A deterministic approximation of \eqref{e:StateTCL_dist} was constructed for comparison, in which $D_{k+1}$ is replaced by its mean:
\begin{equation}
	\theta_{k+1} = \theta_k + \alpha [\bar{\theta}^a - \theta_k] - \beta U_k
	\label{e:StateTCL_prt}
\end{equation}
with $\bar{\theta}^a =\theta^a + \tilde{\theta}^a$ with $\tilde{\theta}^a = \bar{d} \epsy / \alpha$. 

Optimal policies were calculated for each of three models: the stochastic model \eqref{e:StateTCL_dist},  its deterministic approximation
\eqref{e:StateTCL_prt},   and the cruder deterministic approximation obtained on setting $D_{k+1}\equiv 0$ in  \eqref{e:StateTCL_dist}   (equivalently, 	\eqref{e:StateTCL_prt} with  $\bar{\theta}^a =\theta^a $). Each policy was then tested on the stochastic model \eqref{e:StateTCL_dist}.

 \Cref{f:eps5+eps100} displays the results from these experiments, where in each plot
\begin{itemize}
	\item $y_k^\text{ref}$ is the reference signal
	\item $y_k$ is the policy that is optimal for the stochastic model
	\item $\tilde{y}_k$ is the policy that is optimal for \eqref{e:StateTCL_prt}
	\item $\bar{y}_k$ is the policy that is optimal  for \eqref{e:StateTCL_prt} using $\bar{\theta}^a =\theta^a $.
\end{itemize}
The accurate tracking $y_k \approx y^{\textrm{ref}}_k$ is expected because this reference signal is feasible, and  $\kappa>0$ was chosen to be large.
 
It is seen in \Cref{f:eps5+eps100}~(a) that all four trajectories are nearly identical for the smaller disturbance. The deviation is far greater in (b), for which the disturbance is   greater. However, $y_k$ and $\bar{y}_k$ are nearly identical for about the first 30 minutes. This suggests that a deterministic approximation, combined with model predictive control, may be used in place of the stochastic model.

\begin{figure}[h]
	\centering
	\includegraphics[width=.95\linewidth]{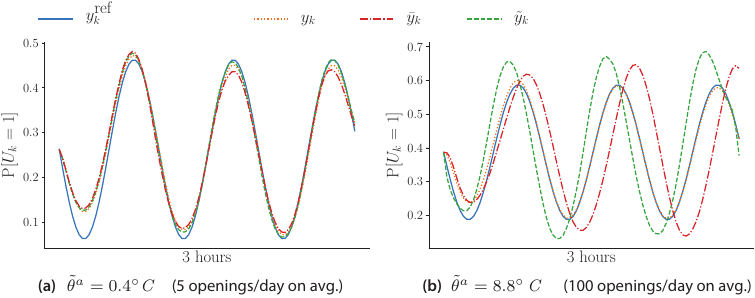}	
	\caption{\small
	Sensitivity experiments support the use of MPC with a deterministic approximation for randomness from nature.}
	 \label{f:eps5+eps100} 
\end{figure}

\subsection{Information architectures} 
\label{s:con_arc}

The choice of information architecture is an interesting topic for future research.   Here are three possibilities:  
\begin{romannum}
	\item \textit{Smart BA}: The BA uses the reference signal $\{r_k\}$ and its estimate of $\nu_0^0$ to compute $\lambda^\star$ and broadcast it to the loads.
	
	\item \textit{Smart Load}: The BA broadcasts $\{r_k\}$ to the loads. Each load computes $\lambda^\star$ based on its internal model and $\nu_0^0 = \delta_{x_0}$, with $x_0 \in \state$ its current state.
	
	\item \textit{Genius Load}: The BA broadcasts $\{r_k\}$ to the loads. Each load computes $\lambda^\star$ based on its internal model and its estimate of $\nu_0^0$.
\end{romannum}

Each approach has its strengths and weaknesses. Approaches (i) and (iii) require knowledge of the initial marginal pmf of the population, $\nu^0_0$. If a perfect estimate is assumed, then the total cost in cases (i) and (iii) is equal to $J^\star(\nu_0^0)$. But, how can a load estimate the marginal pmf of the population? Recall the coupling shown in \Cref{f:hist-evo}:
	 the histograms are nearly identical after about three hours, regardless of the initial condition. If enough time has passed since the latest MPC iteration,  the pmfs $\{\nu_k \}$ computed locally can be used to approximate the marginal pmf of the population (perhaps smoothed using the techniques of \cite{YueChenThesis16,chebusmey17a}).  

In contrast, the total cost for case (ii) is the sum,
\[
\sum_{i=1}^d \nu_0^0(x^i)  J^\star(\delta_{x^i})
\]
since each load optimizes according to its own initial state, $x^i$. Even when the aggregate can easily track $\{r_k\}$, the cost $J^\star(\delta_{x^i})$ may be very large for individuals that are at odds with the reference signal. For example, an increase in power consumption could be requested while a water heater is near its upper temperature limit and must turn off. 
So,  it is possible that approach (ii) will impose greater stress on the loads as compared to the other two options, or will lead to reduced capacity.

\section{Conclusions}
\label{s:conc}


The paper provides a complete theory for KLQ and infinite-horizon counterparts, without the restriction to deterministic dynamics imposed in  \cite{cambusjimey19,cheche17b}.
Plans for future research   include:
\begin{romannum}

	\item   Monte-carlo approaches for both KLQ and IPD-Q.     The approximation \eqref{e:SmallSignalPolicy} invites actor critic methods for approximating the best coefficients  $\{K^\star_{i,j}  , G^\star_i \}$ based on training data with non-constant reference signal,   rather than a small signal approximation.


	\item Evaluate robustness and sensitivity to other types of modeling error.

	\item Investigate alternative transform techniques.
	
	
	\item Consider other cost functions,  such as the Wasserstein distance.   Preliminary results are summarized in \cite{corbusmey23}.
	
	\item Investigate the relationship between optimality and coupling of the pmfs, and the implications to control design.
	
	\item  Careful design of a terminal cost function may result in better performance for smaller time horizons \cite{chemey99a}.

	\item  How is the relative value function $\clH$ appearing in \eqref{e:KLQACOE} related to $h^\lambda$ appearing in \eqref{e:ACOE-IPDQ}  (with $\lambda=\lambda_r$)?

\end{romannum}

\appendix

 \bigskip

\centerline{\bf \Large Appendix}
\section{Convexity} \label{app:cvx}

The following is one step in the proof of \Cref{p:convexity}.
\begin{lemma}
\label[lemma]{t:clDconvex}
The function \eqref{e:KLrate} can be expressed as a convex function of the marginal:
\[
\clD(\nu_k ,\nu_k^0)  =  D(   \nu_k  \|  \nu_k^d)
\]
where $D$ denotes relative entropy, and  $\nu_k^d$   a linear function of $\nu_k$,
\[
\nu_k^d (\st,u) = \fee_k^0(u\mid \st)   \sum_{u'}  \nu_k(\st,u' )
\]
\end{lemma}

\begin{proof}
Let $\spmf_k$ denote the distribution of $S_k$; that is,  $\spmf_k(\st) =  \sum_{u'}  \nu_k(\st,u' )$ for each $\st\in\S$.      Bayes' rule gives $\nu_k(\st,u) = \fee_k(u\mid \st) \spmf_k(\st)$  for each $\st$, $u$, and the desired identity follows:
\[
\clD(\nu_k ,\nu_k^0)  =
 \sum_{\st,u}   \nu_k(\st,u) \log\Bigl(\frac{\fee_k(u\mid \st) \spmf_k(\st) }{ \fee_k^0(u\mid \st) \spmf_k(\st) }\Bigr) =D(   \nu_k  \|  \nu_k^d)
\]
The function $\clD$ is convex because the relative entropy $D$ is jointly convex in its two arguments  \cite{demzei98a} or \cite[Section 1.5]{pri10}.
\end{proof}

\begin{proof}[Proof of \Cref{p:convexity}]
Convexity of the objective  \eqref{e:primal} in the variables ($\nu_k$, $\gamma_k$)  holds because it is the sum of convex functions  (convexity in the marginals $\{ \nu_k\}$ is established in \Cref{t:clDconvex}).

	It is shown next that the constraint \eqref{e:dynamics_constraint} characterizes the dynamics \eqref{e:mfm}. By definition, the constraint \eqref{e:dynamics_constraint} can be expressed as
	$$
	\Prob\{S_k = \st'\} = \sum_{\st,u} \Prob\{S_{k-1} = \st, U_{k-1}=u\} \Prob\{S_k = \st' \mid S_{k-1} = \st, U_{k-1}=u\}
	$$
	Multiplication by $\fee_k(u' \mid \st')$ yields
	\begin{equation}
		\begin{aligned}
			\Prob\{ &S_k=\st', U_k=u'\} 
			\\
			&= \sum_{\st,u} \Prob\{S_{k-1} = \st, U_{k-1}=u\} \Prob\{S_k = \st', U_k=u' \mid S_{k-1} = \st, U_{k-1}=u\}
		\end{aligned} 
	\end{equation}
	which is identical to \eqref{e:mfm}. The proof of the implication \eqref{e:mfm} $\Longrightarrow$ \eqref{e:dynamics_constraint} is direct.  
\end{proof}

\section{Convex duality of relative entropy}
\label{app:A}
 
The proofs of \Cref{t:dualFunctional} and \Cref{p:algo} make use of the following four lemmas. The first is based on a well known result regarding relative entropy. For any function $h\colon \state^{K+1} \to \Re$ denote
\begin{equation}
\L^0(h)\eqdef \sup_{\bfmp} \bigl \{ \langle \bfmp, h \rangle - D(\bfmp\|\bfmp^0) \bigr\}
\label{eq:dualD}
\end{equation}

\begin{lemma}[Convex dual of relative entropy]
\label{lm:dualkullback} 
For each $\bfmp^0\in \clS(\state^{K+1})$ and function $h \colon \state^{K+1} \to \Re$, the (possibly infinite) value of \eqref{eq:dualD} coincides with the log moment generating function:
$$
\L^0(h) = \log \lgl \bfmp^0, e^h \rgl 
$$
Moreover, provided $\L^0(h) <\infty$, the supremum in \eqref{eq:dualD} is uniquely attained with $\bfmp^\star = \bfmp^0\exp(h-\L^0(h))$. That is, the log-likelihood $L^\star = \log(d\bfmp^\star /d\bfmp^0)$ is given by
$$
L^\star(\vx) = h(\vx) - \Lambda^0(h)
$$
\qed 
\end{lemma}

\begin{lemma}
\label{lm:dual_lambda_h}
The dual function can be expressed
\begin{equation}
	\varphi^\star(\lambda,g) = \lambda^T \har - \frac{1}{2\kappa} \|\lambda\|^2 - \lgl \nu_0^0, \Tg^\lambda_1 \rgl + \sum_{k=1}^{K} \min_\st \Big[g_k(\st) - \clT^\lambda_k(g_{k+1};\st) \Big]
\label{e:DualFormulaFinal}
\end{equation}
where $g_{K+1} \equiv 0$.
\end{lemma}

\begin{proof}
First, make the substitution  $\nu_k(\st,u) = \spmf_k(s) \fee_k(u \mid \st)$,
so that the Lagrangian \eqref{e:Lagrangian} can be written
\begin{equation}
\begin{aligned}
\quad
& \clL(\nu, \gamma, \lambda , g) =
\sum_{n=1}^N \Bigl(\frac{\kappa}{2} \gamma_n^2 + \lambda_n \gamma_n + \lambda_n \har_n \Bigr) - \sum_{\st,u} \nu^0_0(\st, u) \sum_{\st'} T_0(x, \st') g_1(\st')
\\
& +
\sum_{k=1}^K \sum_{\st} \spmf_k(\st) \sum_u \fee_k(u \mid \st)  \Big(L_k(\st,u) 
			- \sum_{\st'} T_k(x, \st') g_{k+1}(\st') - \chlambda_k \outpt(\st,u) \Bigr)
\\
& +
\sum_{k=1}^K \sum_{\st} \spmf_k(\st) g_k(s)
\end{aligned}
\label{e:primalLagrangianA}
\end{equation}
with $g_{K+1} \equiv 0$,   and $L_k(\st,u) = 
\log \frac{\fee_k(u \mid \st)}{\fee^0_k(u \mid \st)}
$.
This amounts to a Lagrangian  decomposition since the minimization of the Lagrangian is equivalent to solving $K$ separate convex programs to obtain each of the minimizers $\{\nu_k^{\lambda,g}: \nu_k^{\lambda,g}(\st,u) = \spmf^{\lambda,g}_k(\st) \fee^{\lambda,g}_k(u \mid \st), \, (\st,u) \in \state, \, 1 \le k \le K\}$.   That is, $ \displaystyle\argmin_\fee \clL = $
\begin{equation}
\bigg\{\argmin_{\fee_k :1\leq k\leq K } \sum_{u}  \fee_k(u \mid \st) \bigg[ L_k(\st,u)  - \sum_{\st'} T_k(x, \st') g_{k+1}(\st') - \chlambda_k \outpt(\st,u) \bigg]  \bigg\} 
\end{equation}
It follows from \Cref{lm:dualkullback} that the minimizer is given by
\begin{equation}
\begin{aligned} 
\fee_k^{\lambda,g} (u\mid \st) &= 
\fee_k^0 (u\mid \st) \exp\bigl( \sum_{\st'} T_k(x, \st') g_{k+1}(\st') + \chlambda_k \outpt(\st,u) - \Lambda_k(\st) \bigr)   
\\
\textit{with} \qquad
\Lambda_k(\st) &= \clT^\lambda_k(g_{k+1};\st)
\end{aligned}
\label{e:phi_lemma}
\end{equation}

\Cref{lm:dualkullback} also gives the value:
\begin{equation*}
\begin{aligned} 
  \min_{\fee_k}
\sum_{u}  \fee_k(u \mid \st)  & \Bigl[  L_k(\st,u)  
					- \sum_{\st'} T_k(x, \st') g_{k+1}(\st') - \chlambda_k \outpt(\st,u) \Bigr] = - \clT^\lambda_k(g_{k+1};\st)
\end{aligned} 
\end{equation*}
resulting in
\begin{equation}
\begin{aligned} 
\min_{\nu}
\clL(\nu, \gamma,\lambda, g) =& \sum_{n=1}^N \Bigl(\frac{\kappa}{2} \gamma_n^2 + \lambda_n \gamma_n + \lambda_n \har_n \Bigr) - \sum_{\st,u} \nu^0_0(\st,u) \sum_{\st'}T_0(x, \st')g_1(\st')
\\ &+
\sum_{k=1}^{K} \min_{\spmf_k} \lgl \spmf_k, g_k - \clT^\lambda_k(g_{k+1}) \rgl
\end{aligned}
\label{e:dual_lambda_h}
\end{equation}
Next, observe that the minimizer $\spmf^{\lambda,g}_k$ is obtained when the support of each $\spmf_k$ satisfies
$$ \text{supp} \big(\spmf_k(s) \big) \subseteq \argmin_\st \Big[g_k(\st) - \clT^\lambda_k(g_{k+1};s) \Big]
$$
so that
\begin{equation*}
\min_\st \Big[g_k(\st) - \clT^\lambda_k(g_{k+1};\st) \Big] = \lgl \spmf^{\lambda,g}_k, g_k - \clT^\lambda_k(g_{k+1}) \rgl
\end{equation*}
We also have $\gamma^\lambda_n = - \frac{1}{\kappa}\lambda_n$
Substituting the minimizers $\{\nu^{\lambda,g}_k, \gamma^\lambda_n\}$ into \eqref{e:dual_lambda_h}, and applying \eqref{e:Gamma}, results in \eqref{e:DualFormulaFinal}.
\end{proof}

\section{Duality}
\label{app:B}

\begin{lemma}\label{lm:dual_lambda}	
The maximum of the dual function over $g$ is 
\begin{equation}
\varphi^\star(\lambda) \eqdef \max_g \varphi^\star(\lambda,g) = \lambda^T \har
-\frac{1}{2\kappa} \|\lambda\|^2 
-  \lgl \nu_0^0, \Tg^\lambda_1 \rgl
\label{e:dual_lambda}
\end{equation}
with $ \Tg^\lambda_1 (x) = \sum_{\st'} T_0(x, \st') g^\lambda_1(\st')$.
A maximizer $g^\lambda$ is given by the recursive formula:
\begin{equation}
g^\lambda_k = \clT^\lambda_k(g_{k+1}^\lambda) \,, \;\; 1 \le k \le K \,, \;\; \text{\it where} \,\, g^\lambda_{K+1} \equiv 0 \,,
\label{e:g_lambda}
\end{equation}
\qed 
\end{lemma}

\begin{proof}[Proof of \Cref{lm:dual_lambda}]
Adding a constant to any of the $(g_1, g_2, \dots, g_K)$ does not change the value of $\clL$ \eqref{e:Lagrangian} or $\varphi^\star$ \eqref{e:dualAlambda}, so without loss of generality we assume,
\begin{equation}
\min_\st \Big[ g_k(\st) - \clT^\lambda_k(g_{k+1};\st) \Big] = 0 \;\; \text{\it for each k,}
\label{e:mins}
\end{equation}
and consequently
\begin{equation}
g_k \ge \clT^\lambda_k(g_{k+1}) \;\; \text{\it for each k} \,.
\label{e:ineq}
\end{equation}
Thus, in view of \eqref{e:DualFormulaFinal},
\begin{equation}
\varphi^\star(\lambda) = \lambda^T \har
-\frac{1}{2\kappa} \|\lambda\|^2 
- \min_{g_1} \sum_{\st,u} \nu^0_0(\st,u) \sum_{\st'}T_0(x, \st')g_1(\st') \,,
\label{e:varphi*pf}
\end{equation}
where the minimum is subject to the constraint \eqref{e:ineq}. Next, observe that $\clT^\lambda_k$ is a monotone operator, so that for each $k\le K$,
$$
g_k \ge \clT^\lambda_k \circ \clT^\lambda_{k+1} \circ \dots \circ \clT^\lambda_K(g_{K+1}) \doteq g^\lambda_k \,, \quad \text{\it where} \,\, g_{K+1} \equiv 0
$$
Based on the expression \eqref{e:varphi*pf}, we now show that the maximum $\argmax_g \fee^\star(\lambda,g)$ is obtained by choosing each $g_k$ to reach this lower bound, giving \eqref{e:g_lambda}. Indeed, $g^\lambda_1$   achieves the minimum in \eqref{e:varphi*pf},  since $g^\lambda_1 \le g_1$ for any $g_1$ for which \eqref{e:ineq} holds. This result along with \eqref{e:mins} yields \eqref{e:dual_lambda}.
\end{proof}

\begin{lemma}\label{lm:g_bound}	
The maximizers $\{g^\lambda_k\}$ have at most linear growth in $\|\lambda\|$:
\begin{equation}
|g^\lambda_k(s)| \leq C_k \|\lambda\| \,, \quad 1 \le k \le K
\label{e:g_bound}
\end{equation}
where $C_k = \|\outpt\|_\infty \sum_{i=k}^{K} \|w(i)\|$ and $w(i)$ is the vector $\{w_1(i), w_2(i), \dots, w_N(i)\}$.
\qed 
\end{lemma}

\begin{proof}[Proof of \Cref{lm:g_bound}]
The proof is by induction, starting with   $k=K$:
\begin{equation}
\begin{aligned}
\bigl|
g^\lambda_K(s)
\bigr| 
  &=
\bigl|
 \clT^\lambda_K(g_{K+1};\st)
 \bigr| 
\\ &=
\Big|
\log\Bigl( \sum_u \fee_K^0(u\mid \st) \exp\bigl(\chlambda_K \outpt(\st,u) \bigr) \Bigr)
\Big|
\\&\leq
\Big|
\log\Bigl( \sum_u \fee_K^0(u\mid \st) \exp\bigl(|\chlambda_K| \|\outpt\|_\infty \bigr) \Bigr)
\Big|
\\&=
\big|
\log\Bigl( \exp\bigl(|\chlambda_K| \|\outpt\|_\infty \bigr) \Bigr)
\big|
\leq
 C_K \|\lambda\|
\eqdef
 \|w(K)\| \|\outpt\|_\infty \|\lambda\|  \,,
\end{aligned}
\end{equation}
which establishes the induction hypothesis for $K$. Now, assume the hypothesis is true for $k \le K$. Then,
\begin{equation*}
\begin{aligned}
\big|
g^\lambda_{k-1}(s)
\big|
 &=
 \big|
 \clT^\lambda_{k-1}(g_k;\st)
 \big|
\\ &=
\Big|
\log\Bigl(\sum_u \fee_{k-1}^0(u\mid \st) \exp\bigl(\sum_{\st'}T_{k-1}(x, \st')g_k(\st') + \chlambda_{k-1} \outpt(\st,u) \bigr) \Bigr)
\Big|
\\&\leq
\Big|
\log\Bigl( \sum_u \fee_{k-1}^0(u\mid \st) \exp\bigl(\sum_{\st'}T_{k-1}(x, \st') C_k \|\lambda\| + |\chlambda_{k-1}| \|\outpt\|_\infty \bigr) \Bigr)
\Big|
\\&\leq
C_{k-1} \|\lambda\|
\eqdef
\bigl[
 C_k+   \|w(k-1)\| \|\outpt\|_\infty  \bigr]  \|\lambda\| 
 \end{aligned}
\end{equation*}
This completes the proof of \eqref{e:g_bound} by induction.
\end{proof}

\begin{proof}[Proof of \Cref{t:dualFunctional}]
	We prove the existence of a maximizer $\lambda^\star$ by showing that $\fee^\star(\lambda)$ is an anti-coercive function, i.e., $\fee^\star(\lambda) \rightarrow -\infty$ as $\|\lambda\| \rightarrow \infty$. 
	By \Cref{lm:g_bound}, there exists $C_1<\infty$ such that
	\begin{equation*}
	\begin{aligned}
	\varphi^\star(\lambda) &= \lambda^T \har
	-\frac{1}{2\kappa} \|\lambda\|^2 - \sum_{\st,u} \nu^0_0(\st,u) \sum_{\st'}T_0(x, \st')g^\lambda_1(\st')
	\\ & \leq
	\|\lambda\| \|\har\| - \frac{1}{2\kappa}\|\lambda\|^2 +   \max_{ \st' } |g^\lambda_1(\st')|
	\\ & \leq
	\|\lambda\| \|\har\| - \frac{1}{2\kappa}\|\lambda\|^2 + C_1 \|\lambda\|
	\end{aligned}
	\end{equation*}
	Since $\fee^\star(\lambda)$ is upper-bounded by an anti-coercive function, $\fee^\star(\lambda)$ itself is an anti-coercive function.
	Thus a maximizer $\lambda^\star$ exists,   and $(\lambda^\star, g^\star) = (\lambda^\star, g^{\lambda^\star})$ by \eqref{e:g_lambda}.

	
	The primal is a convex program, as established in \Cref{p:convexity}. To show that there is no duality gap it is sufficient that Slater's condition holds \cite[Section~5.3.2]{boyvan04a}. This condition holds: the relative interior of the constraint-set for the primal is non-empty since it contains $\{\nu_k^0\}$. Optimality of \eqref{e:phi-h-Lambda} is established by substituting $g_{k+1}^\star$ into \eqref{e:phi_lemma} and by making the substitution $g^\star_k = \clT^\lambda_k(g^\star_{k+1})$ implied by \eqref{e:g_lambda}.
\end{proof}

\begin{proof}[Proof of \Cref{p:algo}] This proof has three parts:
	\begin{romannum}
	\item \Cref{e:g_lambda2} is proven by \Cref{lm:dual_lambda}.
	
	\item \Cref{e:dualAlambda} is proven by \Cref{lm:dual_lambda}.
	
	\item The representation of the derivative in part (iii) is standard (e.g., Section~5.6 of \cite{boyvan04a}), but we provide the proof for completeness. The representation \eqref{e:Lagrangian} implies that $\varphi^\star$ is concave in $(\lambda, g)$, since it is the infimum of linear functions. This representation also gives a formula for a derivative:
	$$
	\frac{\partial}{\partial\lambda_n}    \varphi^\star(\lambda,g) = \har_n - \frac{1}{\kappa}  \lambda_n - \sum_{k=1}^{K} w_n(k) \lgl  \nu_k^{\lambda,g} , \outpt \rgl \,, \qquad 1\le n\le N
	$$
	where $\nu_k^{\lambda,g} $ is any optimizer in \eqref{e:dual_lambda_h}.  Using  $\varphi^\star(\lambda) = \varphi^\star(\lambda,g^\lambda)$ then gives
	$$
	\frac{\partial}{\partial\lambda_n} \varphi^\star(\lambda) = \har_n - \frac{1}{\kappa}  \lambda_n - \sum_{k=1}^{K} w_n(k) \lgl \nu_k^{\lambda}, \outpt \rgl + \frac{\partial}{\partial g} \varphi^\star\, (\lambda,g^\lambda) \cdot 	\frac{\partial}{\partial\lambda_n} g^\lambda
	$$
	The first order condition for optimality gives $\frac{\partial}{\partial g} \varphi^\star\, (\lambda,g^\lambda) = 0$, which completes the proof of the representation. It is evident that $ \varphi^\star$ is continuously differentiable since $\nu^{\lambda}_k $ is continuously differentiable for each $k$ by construction.
	\end{romannum}
\end{proof}

\begin{proof}[Proof of \Cref{c:optLLR}] This proof has three parts:
	\begin{romannum}
	\item Application of \eqref{e:LLR} and \eqref{e:phi-h-Lambda} results in the log-likelihood ratio:
	\[
	\begin{aligned}
	L(\vx) &= \sum_{k=1}^{K} \Big( \sum_{\st}T_k(x, \st)g^\lambda_{k+1}(\st) + \chlambda_k \outpt(x_k) - g^\lambda_k(\st_k) \Big)
	\\ &=
	 \sum_{k=1}^{K} \Big( \Tg^\lambda_{k+1}(x_k) + \chlambda_k \outpt(x_k) - g^\lambda_k(\st_k) \Big) 
	\end{aligned} 
	\]
	where the second identity follows from the definition  \eqref{e:Gamma}.   We have from the definitions,
	$\Tg^\lambda_{K+1}   \equiv 0$, which results in 
	\[
	L(\vx) =
	-
	\Tg^\lambda_{1}(x_0) +
	 \sum_{k=1}^{K} \Big( \Tg^\lambda_{k}(x_{k-1}) + \chlambda_k \outpt(x_k) - g^\lambda_k(\st_k) \Big) 
	\]
	This combined with \eqref{e:delta} yields \eqref{e:llr}.
	
	\item Applying the definition of relative entropy as the mean log-likelihood, and noticing that $\Expect_{p^\lambda} \bigl[ \Delta_{k}(X_{k-1},S_{k}) \bigr] = 0$ for $1 \le\ k \le K$, results in
	\begin{equation*}
	\begin{aligned}
	\sum_{\vx} p^\lambda(\vx) L(\vx)   &=
	\sum_{\vx} p^\lambda(\vx) \Bigg( \sum_{k=1}^{K} \chlambda_k \outpt(x_k) - \Tg^\lambda_1(x_0) \Bigg)
	\\ &=
	\sum_{k=1}^{K} \sum_{x_k} \sum_{x_i, i \neq k} p^\lambda(\vx) \chlambda_k \outpt(x_k) - \sum_{x_0} \sum_{x_i, i \neq 0} p^\lambda(\vx) \Tg^\lambda_1(x_0)
	\\ &=
	\sum_{k=1}^{K} \chlambda_k \lgl \nu^\lambda_k, \outpt \rgl - \lgl \nu_0^0, \Tg^\lambda_1 \rgl
	\end{aligned}
	\end{equation*}
	
	\item Substitution of \eqref{e:rel_ent} into \eqref{e:primal3} results in \eqref{e:primal4}.
	\end{romannum}
\end{proof}

\section{IPD-Q} 
\label{app:IPDQ}

 \begin{proof}[Proof of \Cref{t:LQRbLemma}] 
 An application of the implicit function theorem   tells us that $\{r^\lambda, \fee^\lambda :  \lambda\in \Re\}$ are smooth as functions of $\lambda$, whose derivatives may be expressed
 \[
 \begin{aligned}
\frac{d}{d\lambda}  r^\lambda &=    \frac{d}{d\lambda} \lgl \pi^\lambda, \outpt \rgl +  1/\kappa
				=    \varsigma^2_\lambda +  1/\kappa
   \\
    \frac{d}{d\lambda}  \log\bigl(\fee^\lambda(u\mid \st)   \bigr)      & = \barH^\lambda(x) + \outpt(x)  -   \frac{d}{d\lambda}  \Gamma^\lambda(\st)  
\end{aligned} 
 \]
The first identities follow from \eqref{e:rCS} and then \eqref{e:PoissonElementsY}.   The formula for the derivative of 
$ \log(\fee^\lambda)$ is immediate from \eqref{e:phistarIPDQ}.

The proof of \eqref{e:tilfee} requires approximations for $r^\lambda$ and $\fee^\lambda$ in a neighborhood of zero.
The first approximation is gven by
 $ r^\lambda  =   \bigl(  \varsigma^2_0 + 1/ \kappa \bigr)\lambda + O(\lambda^2)$.
The definition \eqref{e:Gamma-lambda} implies that 
\[
 \begin{aligned}
    \frac{d}{d\lambda}  \Gamma^\lambda(\st)  \big|_{\lambda=0}  &=   \barH^0_\st + \outpt_\st \,,
\\[.5em]
\textit{which  gives,
}\quad \log\bigl(\fee^\lambda(u\mid \st)   \bigr) & =  \log\bigl(\fee^0(u\mid \st)   \bigr)  + \Uplambda(x) \lambda + O(\lambda^2)
\end{aligned}  
 \]

An inversion is applied to express $\lambda$ as a function of $r$, giving
 \[
  \fee (u\mid s,r)  =    \fee^0 (u\mid \st)\exp\bigl(  \Uplambda(x) \lambda_r \bigr)   + O(r^2)
 \]
 with $\lambda_r =\bigl(  \varsigma^2_0 + 1/ \kappa \bigr)^{-1} r  +O(r^2)$.     Hence \eqref{e:tilfee} follows from  a first order Taylor series approximation of the exponential.
 \end{proof}

\begin{proof}[Proof of \Cref{t:IPDQ-LQRb}]
The proof of \Cref{t:IPDQ-LQR} is adopted, with one significant change: the Taylor series approximation is avoided,  and instead the bilinear structure is used,  $\nu_{k+1} = \nu_k P_k $,  in which  $P_k = P_{\fee_k}$.    On adding, subtracting, and rearranging terms, 
\[
\nu_{k+1} =  
\nu_k P^0     +   \pi^0  [P_k -P^0]     
+
[ \nu_k -\pi^0]  [P_k -P^0]   
\]
 Using invariance of $\pi^0$, and the definition $\tilnu_k =  \nu_k -\pi^0$, 
   gives the error recursion,
\[
\tilnu_{k+1} =     \tilnu_k P^0     +   \pi^0  [P_k -P^0]     
+
 \tilnu_k    [P_k -P^0]   
\]
This is identical to \eqref{e:mfmXr} through notation.  In particular, the quadratic term $ \tilnu_k    [P_k -P^0]   $ evaluated at $x^i$ is precisely $\epsy_k^i$.
Similarly, 
 \[
  \pi^0  [P_k -P^0] (x^i)  =   \sum_x  \pi^0(x) T(x, \st^i)  \tilfee_k(u^i\mid \st^i)  =  \spmf^0(\st^i)  \tilfee_k(u^i\mid \st^i)   \,,
 \]
and recall that $ \spmf^0(\st^i) $ is the $i$th diagonal element of $B$.
The right hand side is approximated using \Cref{t:LQRbLemma} to complete the proof.
 \end{proof}

\bibliographystyle{siamplain}

\def\cprime{$'$}\def\cprime{$'$}


\end{document}